\documentclass[11pt]{article}

\usepackage{layout}
\usepackage[makeroom]{cancel}
\usepackage{bbm}
\usepackage[affil-it]{authblk}

\usepackage{amsfonts}
\usepackage{amsmath}
\usepackage{amssymb}
\usepackage{amsthm}
\usepackage{graphicx} \usepackage{enumerate} \usepackage{multicol}
\usepackage{mathrsfs} \usepackage[all,cmtip]{xy}
\usepackage{enumerate}
\usepackage{xcolor}
\usepackage{cite}
\usepackage[textwidth=3.2cm]{todonotes}

\usepackage[colorlinks,linkcolor=blue,citecolor=blue,pagebackref,hypertexnames=false, breaklinks]{hyperref}

\usepackage{float}

\newlength{\bibitemsep}
\newlength{\bibparskip}
\let\oldthebibliography\thebibliography
\renewcommand\thebibliography[1]{%
  \oldthebibliography{#1}%
  \setlength{\parskip}{\bibitemsep}%
  \setlength{\itemsep}{\bibparskip}%
}

\newtheorem{theorem}{Theorem}[section]

\newtheorem{definition}[theorem]{Definition}

\newtheorem{corollary}[theorem]{Corollary}

\newtheorem{remark}[theorem]{Remark}
\newtheorem{example}[theorem]{Example}
\newtheorem{examples}[theorem]{Examples}
\newtheorem{foo}[theorem]{Remarks}

\DeclareMathOperator *{\osc}{osc}

\def\vint{\mathop{\mathchoice%
          {\setbox0\hbox{$\displaystyle\intop$}\kern 0.22\wd0%
           \vcenter{\hrule width 0.6\wd0}\kern -0.82\wd0}%
          {\setbox0\hbox{$\textstyle\intop$}\kern 0.2\wd0%
           \vcenter{\hrule width 0.6\wd0}\kern -0.8\wd0}%
          {\setbox0\hbox{$\scriptstyle\intop$}\kern 0.2\wd0%
           \vcenter{\hrule width 0.6\wd0}\kern -0.8\wd0}%
          {\setbox0\hbox{$\scriptscriptstyle\intop$}\kern 0.2\wd0%
           \vcenter{\hrule width 0.6\wd0}\kern -0.8\wd0}}%
          \mathopen{}\int}

\newcommand{\sk}[2]{\langle #1 , #2\rangle}

\begin{document}

\title{The $p$-ellipticity condition for second order elliptic systems and applications to the Lam\'e and homogenization problems}
\author{Martin Dindo\v{s}}
\affil{School of Mathematics\\
        The University of Edinburgh and Maxwell Institute of Mathematical Sciences, UK}

\author{Jungang Li}
\affil{Department of Mathematics\\ 
	Brown University, USA}

\author{Jill Pipher}
\affil{Department of Mathematics\\ 
	Brown University, USA}

\maketitle
 
\begin{abstract} The notion of {\it $p$-ellipticity} has recently played a significant role in improving our understanding of issues of solvability of boundary value problems for scalar complex valued elliptic PDEs. In particular, the presence of $p$-ellipticity ensures higher regularity of solutions of such equations.
  
In this work we extend the notion of  $p$-ellipticity to second order elliptic systems. Recall that for systems, there is no single notion of {\it ellipticity}, rather a more complicated picture emerges with ellipticity conditions of varying strength such as the Legendre, Legendre-Hadamard and integral conditions. A similar picture emerges when $p$-ellipticity is considered. In this paper, we define three new notions of $p$-ellipticity, establish relationships between them and show that each of them does play an important role in solving boundary value problems. 

These important roles are demonstrated by establishing extrapolation results for solvability of the $L^p$ Dirichlet problem for elliptic systems, followed by applications of this result in two different scenarios: one for the Lam\'e system of linear elasticity and another in the theory of homogenization. 
  
\end{abstract}

\section{Introduction}

This paper continues the exploration of a recently discovered structural condition for matrices that plays a key role in the solvability of boundary problems for divergence form elliptic equations associated to matrices with non-smooth coefficients. This condition, $p$-ellipticity, was introduced independently in \cite{DP} and \cite{CD}, and investigated for its role in
two different problems concerning complex valued divergence form operators.  When $p=2$, $p$-ellipticity coincides with the classical  {\it ellipticity} property, essential for the theory of (real) second order partial differential equations. In the previous literature, the $p$-ellipticity condition has been shown to be significant in the study of higher regularity of solutions to complex valued second order divergence form operators (\cite{DP} \cite{DP3}, \cite{FMZ}). In particular, it was shown in \cite{DP2} that the solvability of the Dirichlet problem with boundary data in $L^q$ can be extrapolated from a specific value of $q$ to higher values depending on the range of $p$-ellipticity. This latter fact is remarkable given that there is no maximum principle for complex coefficient equations, which is the easy avenue to extrapolating solvability of the Dirichlet problem in the real coefficient setting.

We formulate here several possible extensions of this condition for second order real elliptic systems: a {\it strong} pointwise condition, an {\it integral} condition, and a {\it weak} pointwise condition that can be compared to the Legendre-Hadamard condition. As was the case for scalar complex coefficient equations, our formulations of $p$-ellipticity for systems of equations are strengthened conditions considered by Cialdea and Maz'ya (see \cite{CM17}) in the context of their work on $L^p$ dissipativity.  See also \cite{C10}, \cite{CM05} and \cite{CM06}.
 
Consider an open subset $\Omega\subset {\mathbb R}^n$. The second order systems we consider may also have lower order terms and can be written in the form:

\begin{equation}\label{eq:dopS}
(\mathcal Lu)_\alpha=\partial_h(A^{hk}_{\alpha\beta}(x)\partial_k u^\beta)+B^h_{\alpha\beta}(x)\partial_hu^\beta,\quad \alpha=1,2,\dots, m;\quad\mbox{for }u:\Omega\to\mathbb C^m,
\end{equation}
where here and in what follows we shall use the Einstein convention summing over repeating indices.
When we write $\nabla$ we shall mean the gradient $(\partial_{x_1},\partial_{x_2},\dots,\partial_{x_n})$. Our indices will usually run as follows: $h,k=1,2,\dots,n$ and $\alpha,\beta=1,2,\dots,m$. Here $n$ stands for the underlying dimension of our domain $\Omega\subset {\mathbb R}^n$ and $m$ denotes the dimension of the  vector valued function $u:\Omega\to{\mathbb R}^m$.\vglue2mm

We shall say that $u:\Omega\to{\mathbb R}^m$ is the weak solution of \eqref{eq:dopS} if the sesqilinear form 
${\mathcal B}$ associated with our equation defined by \eqref{eq:sesq} vanishes for all $v\in C_0^\infty(\Omega;\mathbb C^m)$, i.e., $\mathcal B(u,v)=0$ for all such $v$.\vglue2mm

By way of background, recall that for systems of equations, there are at least three different notions of classical ellipticity of the tensor $A=(A^{hk}_{\alpha\beta})$.

Given a coefficient tensor $A$ with bounded measurable entries defined in $\Omega$, 
$A$ is said to be {\it strongly elliptic} if there exists a constant $C>0$ such that
\begin{equation}\label{EllipL}
\Re e\sk{A(x)\xi}{\xi}_{{\mathbb C}^{n\times m}}=\Re e \left(A^{hk}_{\alpha\beta}(x)\xi_{h}^{\alpha}\overline{\xi_{k}^{\beta}}\right)\ge C|\xi|^2
\end{equation}
for all $\xi=(\xi^{\alpha}_{h})\in{\mathbb C}^{n\times m}$ and a.e. $x\in\Omega$.  Strong ellipticity, \eqref{EllipL},  is traditionally
referred to as the {\bf Legendre} condition. It is the {\bf strongest} form of ellipticity, and when true  
it is usually relatively easy to verify, since it must hold pointwise.  

In particular, it follows from \eqref{EllipL}, via integration, that for any $v\in W^{1,2}_0(\Omega,\mathbb C^m)$ we have:
\begin{equation}\label{EllipIC}
\Re e\int_\Omega \langle A \nabla v , \nabla v \rangle\, dx=\Re e \int_{\Omega}
A^{hk}_{\alpha\beta}(x)\partial_kv^\beta\overline{\partial_hv^\alpha}\,dx\ge C\int_{\Omega}|\nabla v|^2dx.
\end{equation} 
This {\it integral condition} of ellipticity is the starting point for the Lax-Milgram lemma which
 allows one to find weak solutions.

Finally,  the {\bf weakest} form of ellipticity is the 
 {\bf Legendre-Hadamard} condition:
\begin{equation}\label{EllipLH} \Re e\left\langle
(A^{hk}(x)q_h {q_k)}\eta,\eta\right\rangle=
\Re e\left(
A^{hk}_{\alpha\beta}(x)\eta^\alpha \overline{\eta^\beta} q_h {q_k}\right)\ge \lambda |\eta|^{2}|q|^2
\end{equation}
for all $\eta=(\eta^{\alpha})\in{\mathbb C}^{m}$, $q=(q_h)\in{\mathbb R}^{n}$, and a.e.  $x\in\Omega$.

It is always the case that \eqref{EllipIC} implies \eqref{EllipLH}, and all three conditions are equivalent when the operator is scalar ($m=1$) and real-valued. Furthermore, if the coefficients of $A$ are uniformly continuous (or, in 
the case that $\Omega$ is a bounded domain, we only need $ A_{\alpha \beta}^{hk} \in C(\overline{\Omega})$), then (c.f. \cite{Y}) condition \eqref{EllipLH} implies something similar to \eqref{EllipIC}, namely a G{\aa}rding-type integral inequality

\begin{equation}\label{EllipICw}
\Re e\int_\Omega \langle A \nabla v , \nabla v \rangle\, dx+ M\int_{\Omega}|v|^2dx \ge C\int_{\Omega}|\nabla v|^2dx,
\end{equation}
for all $v\in W^{1,2}_0(\Omega,\mathbb C^m)$ and sufficiently large $M=M(A,\Omega)>0$.
\vglue2mm

In this paper, we introduce three new notions of $p$-ellipticity for systems that are analogues of these three classical notions of ellipticity. In particular,
when $p =2$ they coincide with the three classical notions of ellipticity for systems.
Our first main theorem relates the three new properties of $p$-ellipticity introduced in this paper; see Section \ref{two} for 
the precise definitions.

\begin{theorem}\label{p-ellipticity AT} Let $\Omega\subset{\mathbb R}^n$ be open and $A=(A_{\alpha\beta}^{hk}(x)): \Omega\to \mathbb C^{n\times m}$ be a bounded tensor-valued function. Then the following statements hold.
\begin{itemize}
\item[(i)] For any $p\in (1,\infty)$ the strong $p$-ellipticity condition \eqref{EllipLp} implies the integral $p$-ellipticity condition \eqref{eq:pellintp}.
\item[(ii)] For any $p\in (1,\infty)$ the integral ellipticity condition \eqref{eq:pellintp} implies the weak $p$-ellipticity condition \eqref{EllipLHp}.
\item[(iii)] When $p=2$ the condition \eqref{EllipLp} is just the usual strong ellipticity condition \eqref{EllipL}, the condition \eqref{eq:pellintp} is the usual integral ellipticity condition \eqref{EllipIC} and the condition 
 \eqref{EllipLHp} is just the Legendre-Hadamard condition  \eqref{EllipLH}.
\item[(iv)] Given any $p\in (1,\infty)$ if $A$ satisfies one of the conditions  \eqref{EllipLp}, \eqref{eq:pellintp} or 
\eqref{EllipLHp} then $A^*$ satisfy the same $p'$-ellipticity condition. If $A$ satisfies condition 
 \eqref{EllipLHp} for some  $p \in (1, \infty)$ then it also satisfies  \eqref{EllipLHp} for $p=2$, and thus the Legendre-Hadamard condition  \eqref{EllipLH}.
\item[(v)] If $A$ satisfies one of the conditions  \eqref{EllipLp}, \eqref{eq:pellintp} or 
\eqref{EllipLHp} for both $p$ and $q$ from the interval $(1,\infty)$ then $A$ satisfies the same condition 
for any $r$ from an open interval containing the points $p$ and $q$.
 \end{itemize}
 \end{theorem} 
 
 \medskip
 
 The proof of Theorem \ref{p-ellipticity AT} will be given after Definition \ref{LH-p}. We take this opportunity to thank A. Cialdea for pointing out a gap in our original proof of part (ii) of
 this theorem, necessitating a further reduction that we elaborate on in Section \ref{two} in the proof.
 
 \smallskip
 
Each of these three new notions of $p$-ellipticity are significant and are used in this paper for 
new results.  The integral condition is important in two ways. First, it leads to higher integrability of weak solutions - a  
regularity result that comes from a limited Moser iteration argument first introduced in \cite{DP} for scalar equations. Second, it is used in  
extrapolation. We recall that in
the case of scalar second order equations, the $p$-ellipticity condition allows one to extrapolate the range of $q$ for which one obtains solvability of the Dirichlet problem with boundary data in $L^q$ with nontangential maximal function estimates on solutions. We will see that the same can be shown for systems under the assumption of the {\it integral} $p$-ellipticity condition. Properties (iv) and (v) are extensions to the case of systems of analogous properties in the scalar case, observed
in \cite{CD} and \cite{DP}.

In Section \ref{Dir}, we define solvability of the Dirichlet problem for these systems, with boundary data in $L^q$, in the classical sense of nontangential maximal function estimates.
We are then able to claim the following extrapolation property, by an argument very similar to that of \cite{DP3} in the scalar case.
 
 \begin{theorem}[\textbf{Extrapolation}]\label{Extrapolation}
Let $\Omega\subset \mathbb R^n$ be a bounded or unbounded Lipschitz domain. Let
\begin{equation}\label{eq:dopS-II}
(\mathcal Lu)_\alpha=\partial_h(A^{hk}_{\alpha\beta}(x)\partial_k u^\beta)+B^h_{\alpha\beta}(x)\partial_hu^\beta,\quad \alpha=1,2,\dots, m;\quad\mbox{for }u:\Omega\to\mathbb C^m,
\end{equation}
be a second order operator with bounded and measurable coefficients $A$ and $|B|\lesssim \delta(x)^{-1}$, where
$A$ satisfies \eqref{EllipIC} and $\delta(x)$ is the distance of $x$ to the boundary of $\Omega$. Define 
  $$p_0 = \sup \{ p: A \ \text{satisfies condition  \eqref{eq:pellintp}}  \}.$$
  Assume that the $L^q$ Dirichlet problem is solvable for $\mathcal L$ for some $q\in (1,\frac{p_0(n-1)}{(n-2)})$  (if $p_0 = \infty$ or $n=2$ we require $q\in (1,\infty)$).

\noindent Then the $L^p$ Dirichlet problem is solvable for 
$\mathcal L$ for $p$ in the range $[q,  \frac{p_0(n-1)}{(n-2)})$, if one of the following constraints holds on the size of the vector $B$.
\begin{itemize}
\item $\Omega$ is bounded and $B(x)=o(\delta(x)^{-1})$ as $x\to\partial\Omega$.
\item $\Omega$ is bounded and $\limsup_{x\to\partial\Omega} |B(x)\delta(x)|\le K$. Here $K=K(A,p,n)>0$ is sufficiently small. 
\item $\Omega$ is unbounded and $|B(x)\delta(x)|\le K$ for all $x\in\Omega$. Here $K=K(A,p,n)>0$ is sufficiently small. 
\end{itemize}
\end{theorem}
 
We then apply our general results to two well studied cases: the Lam\'e equations, and real elliptic systems of equations with rapidly oscillating periodic coefficients (homogenization). In Section \ref{five}, we investigate the consequences of $p$-ellipticity for the variable coefficient Lam\'e system:

\begin{eqnarray}\label{eq-Lame}
  \mathcal{L} u = \nabla \cdot (\lambda (x) (\nabla \cdot u) I + \mu(x) (\nabla u + (\nabla u)^T)).
\end{eqnarray}

Under certain natural conditions on the Lam\'e coefficients, the $L^2$ Dirichlet problem was shown to be solvable
in \cite[Corollary 1.2]{DHM}. Thus, in light of the extrapolation results,
 it is of interest to consider what further conditions on the coefficients give rise to $p$-ellipticity, for $p>2$. This is the subject of 
 Section \ref{five}, and Theorem \ref{Prop-Lame} in particular.
 For this investigation, we see the significance of using the pointwise $p$-ellipticity conditions, one of which produces ``necessary conditions" on the Lam\'e coefficients, while the other gives ``sufficient conditions". 
 
When specialized to the case of constant coefficients, the results of Section \ref{five} (see  \eqref{eq-LameICMd} in particular) 
yield as a corollary the following improvement on the $L^p$-dissipativity results of \cite{CM17}.
 
 \begin{theorem} If $\lambda,\,\mu$ are constants in $\Omega$,  the operator $\mathcal L$ defined by \eqref{eq-Lame} is $L^p$-dissipative if
\begin{equation}\label{eq-LameICa}
\left(1-\frac2p\right)^2\le 1-\left(\frac{\lambda+\mu}{\max\{\mu,\lambda+2\mu\} +\varepsilon}\right)^2
\end{equation}
for some $\varepsilon=\varepsilon(n,\lambda,\mu)>0$. This improves the range of $L^p$-dissipativity of the Lam\'e operator in dimensions 3 and higher given in Theorem 3.8 of \cite{CM17} which in our notation can be written as
\begin{equation}\label{eq-LameICb}
\left(1-\frac2p\right)^2\le 1-\frac{|\lambda+\mu|}{\max\{\mu,\lambda+2\mu\} }.
\end{equation}
\end{theorem}

This statement follows immediately from \eqref{eq-LameICMd}, and
clearly, \eqref{eq-LameICa} is a larger interval of $p$'s than that of \eqref{eq-LameICb}.
 
A further application to Lam\'e systems of equations is given in Corollary \ref{C:lame}, extending the $L^2$ solvability 
results of \cite{D20}.  In particular, we get an improvement in the range of $L^p$ solvability with the assumption of $p$-ellipticity in all
dimensions  bigger than 3.

In the same spirit, we use the extrapolation results of this paper to extend solvability results of \cite{KS}
 in the theory of homogenization. To be more precise, consider the following system of equations on the Lipschitz domain:
 
 \begin{equation}
    (\mathcal{L}_\epsilon u)_\alpha = \partial_h (A^{hk}_{\alpha \beta}(x / \epsilon) \partial_k u^\beta), \ \epsilon > 0,
 \end{equation}
 with the coefficient matrix $A$ being elliptic, periodic and satisfying certain symmetry and H\"older continuity condition. The solvability of both $L^2$-Dirichlet problem and $L^2$-regularity problem, together with uniformly estimates of the nontangential maximal function, were obtained in \cite{KS}. 
Combining our result with the extrapolation theory established in \cite{S2,S3}, the solvability can be extended to the range $2-\delta < p < \infty$ when $n = 3$ and $2- \delta < p < \frac{2(n-1)}{n-3} + \delta$ when $n \geq 4$, for some small $\delta$. 

When the strong pointwise $p$-ellipticity holds for the tensor $A(x)$, it holds likewise for
the rescaled tensor $A(x/\epsilon)$. We further observe that the extrapolation result of the present paper is independent of scale. Thus we are able to extend the range of uniform estimates in solvability of the Dirichlet problem depending on the 
range of $p$-ellipticity: see Theorem \ref{Homogen} of Section \ref{six}. 
\smallskip

Finally, we would like to thank the anonymous referee for many helpful suggestions and comments that greatly improved the exposition of this paper. 

\section{Several $p$-ellipticity conditions for elliptic systems}\label{two}

Let $\Omega$ be a domain in $\Bbb{R}^n$. Let us recall first the notion of $p$-ellipticity for scalar complex valued divergence form
equations, as introduced in \cite{CD} and \cite{DP}. We say that a complex valued matrix function $A=(A^{hk}(x))$ associated with the second order operator $\mathcal L$ of the form

\begin{equation}\label{eq:dop}
\mathcal Lu=\partial_h(A^{hk}(x)\partial_k u)+B^h(x)\partial_hu\qquad\mbox{for }u:\Omega\to\mathbb C
\end{equation}
is $p$-elliptic 
if there exists a constant $C=C(A,p)>0$ such that, for almost every $x\in\Omega$,
\begin{equation}
\label{eq:pellCD}
\Re e\sk{A(x)\xi}{\xi+|1-2/p|\bar\xi}_{{\mathbb C}^n}\geqslant C|\xi|^2\hskip 40pt \forall\,\xi\in{\mathbb C}^n.
\end{equation}

For $p=2$ this condition is exactly the classical ellipticity, which explains the terminology. The $p$-ellipticity condition in this form was formulated in this way in \cite{CD}, but for the purposes of this paper, the alternative formulation from \cite{DP} will be more useful:
\begin{equation}\label{eq:pellDP}
\langle \Re e\,A\,\lambda,\lambda\rangle+\langle \Re e\,A\,\eta,\eta\rangle+
\left\langle \left(\textstyle\sqrt{\frac{p'}{p}}
\Im m\,A-\sqrt{\frac{p}{p'}}\Im m\,A^t\right)\lambda,\eta\right\rangle
\ge C(|\lambda|^2+|\eta|^2),
\end{equation}
for all $\lambda,\eta\in\mathbb R^n$ and almost every $x\in\Omega$. A change of variable $\xi=\frac{pp'}{2}\lambda$ and then the choice of $\xi=\Re e(|v|^{-1}\overline{v}\nabla v)$, $\eta= \Im m(|v|^{-1}\overline{v}\nabla v)$ followed by integration over $\Omega$ yields the following integral condition (c.f. \cite[Theorem 2.4]{DP}):
\begin{eqnarray}\label{eq:pellint}
&&\Re e\, \int_{\Omega}\Big[\langle A\nabla v,\nabla v\rangle-(1-2/p)\langle(A-A^*)\nabla(|v|),|v|^{-1}\overline{v}\nabla v\rangle\nonumber\\
&&\qquad-(1-2/p)^2\langle A\nabla(|v|),\nabla(|v|) \rangle\Big]\, dx\ge C\int_{\Omega}|\nabla v|^2\,dx,\qquad\mbox{for all $v\in W^{1,2}_0(\Omega)$.}
\end{eqnarray}

It is important to note that, in general, the reverse direction from $\eqref{eq:pellint}$ back to \eqref{eq:pellDP} or \eqref{eq:pellCD}, does not hold. However,  $\eqref{eq:pellint}$ implies a weaker statement, namely that  for almost every $x\in\Omega$,
\begin{equation}
\label{eq:pellCDs}
\Re e\sk{A_s(x)\xi}{\xi+|1-2/p|\bar\xi}_{{\mathbb C}^n}\geqslant C|\xi|^2\hskip 40pt \forall\,\xi\in{\mathbb C}^n.
\end{equation}
Here $A_s$ denotes the symmetric part of the matrix $A$ (c.f. Theorem 1.3 of \cite{CD}). Furthermore,  
\eqref{eq:pellint}$\Leftrightarrow$\eqref{eq:pellCDs} when the distributional divergence of each column of a matrix $(\Im m A)_a$ is zero, by the same theorem. Here $(\Im m A)_a$  denotes the anti-symmetric part of the matrix $\Im m A$. In the case of an operator of the form \eqref{eq:dop} this can be always arranged (at the expense of some extra first order terms) by symmetrizing the imaginary part of $A$. 

Hence, in the case of scalar complex valued operators the pointwise condition \eqref{eq:pellCDs} is essentially equivalent to the integral condition \eqref{eq:pellint}.\vglue2mm

We now turn our attention to the complex valued elliptic systems of the form
\begin{equation}
(\mathcal Lu)_\alpha=\partial_h(A^{hk}_{\alpha\beta}(x)\partial_k u^\beta)+B^h_{\alpha\beta}(x)\partial_hu^\beta,\quad \alpha=1,2,\dots, m;\quad\mbox{for }u:\Omega\to\mathbb C^m.
\end{equation}

As mentioned in the introduction, for systems of equations, there are at least three different notions of classical ellipticity of the tensor $A=(A^{hk}_{\alpha\beta})$.
In analogy, we define three notions of $p$-ellipticity for complex valued elliptic systems and establish relations between them. 

These notions of ellipticity can be defined in the case of real valued elliptic systems as well; that is when the unknown function is $u:\Omega\to{\mathbb R}^m$ and the coefficients $A^{hk}_{\alpha\beta}$, $B^{h}_{\alpha\beta}$ are real-valued. The only difference is that the inner product in the conditions \eqref{EllipL}, \eqref {EllipIC} and \eqref{EllipLH} is over the real field and we only test these conditions  over the space of real valued tensors/functions.
\vglue2mm

\begin{definition}\label{strong-p}
The tensor-valued function $A^{hk}_{\alpha\beta}(x)$ satisfies the  {\bf strong $p$-ellipticity} condition if the following pointwise condition holds.

\begin{eqnarray}\label{EllipLp}
&&\Re e \left\langle  A(x)\left(\xi-\left(1-\frac2p\right)\xi(\omega)\right), \xi+\left(1-\frac2p\right)\xi(\omega)\right\rangle_{\mathbb C^{n\times m}}  =\\\nonumber
&&\Re e \left[A^{hk}_{\alpha\beta}(x)\left(\xi_{h}^{\alpha}-\left(1-\frac2p\right)\omega^\alpha(\Re e\langle \omega,\xi_h\rangle)   \right)\overline{\left(\xi_{k}^{\beta}+\left(1-\frac2p\right)\omega^\beta(\Re e\langle \omega,\xi_k\rangle)\right)}\right]\
\ge C|\xi|^2,
\end{eqnarray}
for all $\xi=(\xi^{\alpha}_{h})\in{\mathbb C}^{n\times m}$, $\omega\in\mathbb C^m$ with $|\omega|=1$ and a.e. $x\in\Omega$. Here, we have introduced the notation that for a tensor $\xi\in \mathbb C^{n\times m}$ and $\omega\in\mathbb C^m$, $\xi(\omega)$ denotes an element of $\mathbb C^{n\times m}$ defined by
\begin{equation}
\xi(\omega)^\alpha_h=\omega^\alpha   \Re e (\omega^\beta\overline{\xi^\beta_h})=\omega^\alpha(\Re e\langle \omega,\xi_h\rangle).
\end{equation}
\end{definition}

Observe that when $p=2$ this coincides with the usual strong ellipticity \eqref{EllipL}. Observe also that $|\xi(\omega)| \leq |\xi|$, a fact that we will
use later on. 

\medskip
We use Definition \ref{strong-p} to formulate a weaker integral condition.
Consider a function $v\in W^{1,2}_0(\Omega,\mathbb C^m)$. In the definition of strong $p$-ellipticity, take $\xi_h^\alpha=\partial_h v^\alpha$ and choose $\omega=\frac{v}{|v|}$. It follows that
\begin{equation}\label{eq-nnn}
\xi(\omega)_h^\alpha =\frac{v^\alpha}{|v|^2}\Re e \left\langle v,\partial_hv\right\rangle\quad\mbox{i.e.},\quad \xi(\omega)=\frac{v}{|v|}\nabla|v|.
\end{equation}

Hence if \eqref{EllipLp} holds, integrating over $\Omega$ gives that for some $C>0$, 
\begin{equation}\label{eq:pellintp}
    \Re e\int_\Omega \left\langle A \left(\nabla v-\left(1-\frac2p\right)\frac{v}{|v|}\nabla|v|\right) , \nabla v +\left(1-\frac2p\right)\frac{v}{|v|}\nabla|v|\right\rangle dx \geq C \int_\Omega |\nabla v|^2\, dx,
  \end{equation}
for all $v\in W^{1,2}_0(\Omega,\mathbb C^m)$. 

\begin{definition}\label{integral-p}
$A^{hk}_{\alpha\beta}(x)$ satisfies the  {\bf integral $p$-ellipticity} condition if  \eqref{eq:pellintp} holds.
\end{definition}

The integral $p$-ellipticity condition will be the key assumption for us in the subsequent sections of this paper.

We note that \eqref{eq:pellintp} is closely related to the notion of $L^p$-dissipativity of second order operators as defined by Cialdea and Maz'ya (see Lemma 4.1 of \cite{CM14}). $L^p$-dissipativity corresponds to having $C=0$ on the righthand side. We now work with \eqref{eq:pellintp} to convert this integral condition into a form that is useful in the theory of regularity of solutions.

\begin{theorem}\label{p-ellipticity lower bound}
  If $A$ satisfies the integral $p$-ellipticity condition \eqref{eq:pellintp}, then there exists $\lambda_p>0$ such that for any $u:\Omega\to\mathbb C$ such that $|u|^{\frac{p-2}{2}}u \in W^{1,2}_{0}(\Omega;\Bbb{C}^m)$,
  
\begin{equation}\label{IC-final}
   \Re e \int_\Omega \langle A \nabla u , \nabla (|u|^{p-2} u) \rangle dx \ge  \lambda_p \int_\Omega |\nabla u|^2 |u|^{p-2}  dx.
\end{equation}
\end{theorem} 

\begin{remark}\label{r-int} Let us remark here on the interpretation of expressions of the form $\nabla(|u|^{\alpha-1} u)$ for $\alpha> 0$. At points where $|u|>0$ the usual chain rule applies. The set $\{x\in\Omega: |u|=0\}$ can be decomposed as
$$\{x\in\Omega: |u|=0\mbox{ and }|\nabla u|=0\}\cup \{x\in\Omega: |u|=0\mbox{ and }|\nabla u|\ne 0\}.$$
For points in $\{x\in\Omega: |u|=0\mbox{ and }|\nabla u|=0\}$ we shall interpret formulas of the form
$|u|^{\alpha-1}\nabla u$ to vanish on this set. The set of points where $\{x\in\Omega: |u|=0\mbox{ and }|\nabla u|\ne 0\}$ has measure zero and hence does not contribute to the value of integrals such as \eqref{IC-final}. 
\end{remark}

\begin{proof}
 Set $v = |u|^{\frac{p-2}{2}} u$ and

\[
  g_\epsilon = (|v|^2 + \epsilon^2)^{1/2}, \ u_\epsilon = g_\epsilon^{\frac{2}{p} - 1} v. 
\]
By following Cialdea-Maz'ya \cite[4.6]{CM14}, when $p \ge 2$ we have 
  
\begin{eqnarray}\nonumber
    &&\lim_{\epsilon \to 0^+} \Re e \int_\Omega \langle A \nabla u_\epsilon , \nabla (|u_\epsilon|^{p-2} u_\epsilon) \rangle\, dx \\
    &=& \Re e\int_\Omega \left\langle A \left(\nabla v-\left(1-\frac2p\right)\frac{v}{|v|}\nabla|v|\right) , \nabla v +\left(1-\frac2p\right)\frac{v}{|v|}\nabla|v|\right\rangle dx \\
    &\geq &C \int_\Omega |\nabla (|u|^{\frac{p-2}{2}} u)|^2\, dx.\nonumber
\end{eqnarray}
The first line above is exactly
\begin{equation}
  \Re e \int_\Omega \langle A \nabla u , \nabla (|u|^{p-2} u) \rangle\, dx,
\end{equation}
while the last line gives us precisely the righthand side of  \eqref{IC-final} with a new constant $\lambda_p$ (see \eqref{eq:zz}).

For $1 < p < 2$, we use a duality argument. To be precise, set $w = |u|^{p-2}u$, so that $u = |w|^{p'-2}w$. Since $|u|^{\frac{p-2}{2}}u \in  W^{1,2}_{0}(\Omega ; \Bbb{C}^m)$, it is easy to verify that $|w|^{p'-2}w \in W^{1,2}_{0}(\Omega ; \Bbb{C}^m)$. Then

\begin{equation}
  \Re e \int_\Omega \langle A \nabla u , \nabla (|u|^{p-2} u) \rangle dx  =\Re e \int_\Omega \langle A^* \nabla w , \nabla (|w|^{p'-2} w) \rangle dx,
\end{equation}
where $A^*$ is the adjoint of $A$. As the proof of Theorem \ref{p-ellipticity AT} (iv) will show, we observe that $A$ satisfies \eqref{eq:pellintp} if and only if  $A^*$ satisfies \eqref{eq:pellintp} for the dual value $p'=p/(p-1)$. Hence following a similar argument as the case $p > 2$, we obtain

\begin{eqnarray}
 \Re e \int_\Omega \langle A^* \nabla w , \nabla (|w|^{p'-2} w) \rangle  dx 
    &\geq & C \int_\Omega |\nabla (|w|^{\frac{p'-2}{2}} w)|^2  dx \\
    &=&C \int_\Omega |\nabla (|u|^{\frac{p-2}{2}} u)|^2  dx.\nonumber
\end{eqnarray}
To finish the proof, it suffices to show for all $p>1$

\begin{equation}\label{eq:zz}
  |\nabla (|u|^{\frac{p-2}{2}} u)|^2 \approx |u|^{p-2}|\nabla u|^2,
\end{equation}
bearing in mind that, as noted in Remark \ref{r-int}, these expressions all occur in the context of an integral. Indeed, when $|u|>0$

\begin{eqnarray}
  &&|\nabla (|u|^{\frac{p-2}{2}} u)|^2 \nonumber\\
  &=& \sum_{1 \leq h \leq n , 1 \leq \alpha \leq m}\left\langle \frac{p-2}{2} |u|^{\frac{p-4}{2}} \nabla_h |u| u^\alpha + |u|^{\frac{p-2}{2}} \nabla_h u^\alpha ,  \frac{p-2}{2} |u|^{\frac{p-4}{2}} \nabla_h |u| u^\alpha + |u|^{\frac{p-2}{2}} \nabla_h u^\alpha \right\rangle \nonumber\\
  &=& \left( \frac{p}{2} -1 \right)^2 |u|^{p-2} |\nabla |u||^2 + |u|^{p-2} |\nabla u|^2 \\
  &&+ \sum_{1 \leq h \leq n , 1 \leq \alpha \leq m} \left( \frac{p}{2} - 1 \right) |u|^{p - 3} \langle \nabla_h |u| u^\alpha , \nabla_h u^\alpha \rangle + \left( \frac{p}{2} - 1 \right) |u|^{p - 3} \langle  \nabla_h u^\alpha , \nabla_h |u| u^\alpha \rangle. \nonumber
\end{eqnarray}
Since

\begin{equation}
  \nabla_h |u| = \Re e \frac{\langle u, \nabla_h u \rangle}{|u|},
\end{equation}
it follows that for the first term of the second last line, we have

\begin{equation}
  0 \leq |u|^{p-2} |\nabla |u||^2 = |u|^{p-2} \left\vert\Re e \frac{\langle u,\nabla u \rangle}{|u|}\right\vert^2 \leq |u|^{p-2} |\nabla u|^2.
\end{equation}
For the last line, we have

\begin{eqnarray}
  &&\sum_{h,\alpha} |u|^{p - 3} \langle \nabla_h |u| u^\alpha , \nabla_h u^\alpha \rangle + |u|^{p - 3} \langle  \nabla_h u^\alpha , \nabla_h |u| u^\alpha \rangle \nonumber\\
  &=& 2 \sum_{h,\alpha} \Re e |u|^{p - 3} \langle \nabla_h |u| u^\alpha , \nabla_h u^\alpha \rangle\nonumber  \\
  &=& 2 |u|^{p-3} \sum_h \nabla_h |u| \Re e \langle u , \nabla_h u \rangle \nonumber\\
  &=& 2 |u|^{p-2} \sum_h \nabla_h |u| \nabla_h |u| \\
  &=& 2 |u|^{p - 2} |\nabla |u||^2 \leq  2 |u|^{p-2} |\nabla u|^2.\nonumber
\end{eqnarray}
\end{proof}

We are now ready to introduce the weakest form of $p$-ellipticity by analogy with the Legendre-Hadamard condition. This condition also comes from generalizing the two dimensional condition formulated in \cite[Theorem 4.2]{CM14}, proven there to be necessary for $L^p$-dissipativity.

\begin{definition}\label{LH-p}
$A^{hk}_{\alpha\beta}(x)$ satisfies the  {\bf Legendre-Hadamard $p$-ellipticity} condition if 
\begin{equation}\label{EllipLHp}
\Re e\left\langle
(A^{hk}(x)q_h {q_k)}\left(\eta-\left(1-\frac2p\right)\omega(\Re e\langle\omega,\eta\rangle)  \right),\eta+\left(1-\frac2p\right)\omega(\Re e\langle\omega,\eta\rangle) \right\rangle \ge C |\eta|^{2}|q|^2,
\end{equation}
for all $\eta,\omega \in{\mathbb C}^{m}$, $q\in{\mathbb R}^{n}$ with $|\omega|=1$, and a.e.  $x\in\Omega$.
\end{definition}

We now turn to the proof of Theorem \ref{p-ellipticity AT}, which 
shows that the condition above is indeed the weakest form of $p$-ellipticity and establishes further relationships among these three $p$-ellipticity conditions.

\begin{proof}  We have already observed that $(i)$ is true - see the argument just after Definition \ref{strong-p}.\vglue2mm

\noindent For $(ii)$, assume that \eqref{eq:pellintp} holds. Pick a point $x_0\in\Omega$ and consider any $0<\varepsilon<\mbox{dist}(x_0,\partial\Omega)$. Let $\psi\in C^\infty_0(B_1,\mathbb C^m)$ where $B_1$ is a unit ball in $\mathbb R^n$ and consider the test function $v(x)=\psi((x-x_0)/\varepsilon)$. Applying \eqref{eq:pellintp} to $v$ we see that
\begin{equation}
    \Re e\int_{\mathbb R^n} \left\langle A(x_0+\varepsilon y) \left(\nabla \psi(y)-\left(1-\frac2p\right)\frac{\psi(y)}{|\psi|}\nabla|\psi|\right) , \nabla \psi(y) +\left(1-\frac2p\right)\frac{\psi(y)}{|\psi|}\nabla|\psi|\right\rangle dy \geq C \int_{\mathbb R^n} |\nabla \psi|^2. \nonumber
  \end{equation}
Letting $\varepsilon\to 0+$ we see that for almost every $x_0\in\Omega$ we have
\begin{equation}\label{IC-mod}
    \Re e\int_{\mathbb R^n} \left\langle A(x_0) \left(\nabla \psi(y)-\left(1-\frac2p\right)\frac{\psi(y)}{|\psi|}\nabla|\psi|\right) , \nabla \psi(y) +\left(1-\frac2p\right)\frac{\psi(y)}{|\psi|}\nabla|\psi|\right\rangle dy \geq C \int_{\mathbb R^n} |\nabla \psi|^2.
  \end{equation}
Recall that we have assumed $\psi\in (C^\infty_0(B_1))^m$, but by dilation we see that \eqref{IC-mod} must hold for any $\psi\in C^\infty_0(\mathbb R^n,\mathbb C^m)$.

Our plan is to proceed as in  \cite[4.10]{CM14} which is an argument that only works in two dimensions. Without loss of generality assume that $x_0=0$ and so $A=A(x_0)$ is fixed. We begin with a dimensional reduction in the spatial variables so that we can use the argument from \cite{CM14}. Our reduction is inspired by, but different from, a dimension reduction argument for the Lam\'e system in \cite[3.7]{CM14}. For $\varepsilon>0$ consider test functions
\begin{equation} 
\psi_\varepsilon(y)=\Psi(y_1/\varepsilon,y_2/\varepsilon)\beta(y_3,y_4,\dots,y_n),
\end{equation}
where $\Psi\in C^\infty_0(\mathbb R^2,\mathbb C^m)$ and $\beta\in C^\infty_0(\mathbb R^{n-2})$ is nonnegative and real-valued. It follows that
\begin{equation} \nonumber
\nabla \psi_\varepsilon(y)=\left(\frac1\varepsilon\nabla_{y_1,y_2}\Psi(y_1/\varepsilon,y_2/\varepsilon)\beta(y_3,\dots,y_n),\Psi(y_1/\varepsilon,y_2/\varepsilon)\nabla_{y_3,\dots,y_n}\beta(y_3,\dots,y_n)\right).
\end{equation}
We use this in \eqref{IC-mod}. Observe that $\frac{\psi_\varepsilon(y)}{|\psi_\varepsilon(y)|}=\frac{\Psi(y_1/\varepsilon,y_2/\varepsilon)}{|\Psi(y_1/\varepsilon,y_2/\varepsilon)|}$ as $\beta$ is real nonnegative. At the same time we also rescale so that $(y_1/\varepsilon,y_2/\varepsilon,y_3,\dots,y_n)=(x_1,x_2,x_3,\dots,x_n)=(x',x'')$. Here $x'=(x_1,x_2)$ and $x''=(x_3,\dots,x_n)$. It follows that
\begin{eqnarray}&&\nonumber
\Re e\int_{\mathbb R^n} \left\langle A\left(\nabla \psi_\varepsilon(y)-\left(1-\frac2p\right)\frac{\psi_\varepsilon(y)}{|\psi_\varepsilon|}\nabla|\psi_\varepsilon|\right) , \nabla \psi_\varepsilon(y) +\left(1-\frac2p\right)\frac{\psi_\varepsilon(y)}{|\psi_\varepsilon|}\nabla|\psi_\varepsilon|\right\rangle dy=\\
=&&\int_{\mathbb R^{n-2}}\beta^2(x'')\,dx''\,\times\\\nonumber
&&\, \Re e\int_{\mathbb R^2} \left\langle {\bf A}\left(\nabla_{x'} \Psi(x')-\left(1-\frac2p\right)\frac{\Psi(x')}{|\Psi|}\nabla_{x'}|\Psi|\right) , \nabla_{x'} \Psi(x') +\left(1-\frac2p\right)\frac{\Psi(x')}{|\Psi|}\nabla_{x'}|\Psi|\right\rangle dx'+\\\nonumber
+&&\hskip-6mm \varepsilon^2\, \Re e\int_{\mathbb R^{n}} \left\langle {\bf B}\left(\Psi(x')\nabla_{x''} \beta-\left(1-\frac2p\right)\Psi(x')\nabla_{x''}\beta\right) , \Psi(x')\nabla_{x''} \beta +\left(1-\frac2p\right)\Psi(x')\nabla_{x''}\beta\right\rangle dx'dx''\\\nonumber
+&&\hskip-6mm \varepsilon\,\hskip2mm \Re e\int_{\mathbb R^{n}} \mbox{ (all mixed terms) }\, dx'dx'' \ge \\\nonumber
=&&C\int_{\mathbb R^{n-2}}\beta^2(x'')\,dx''\,\times \, \int_{\mathbb R^2} |\nabla_{x'}\Psi(x')|^2dx' +C\varepsilon^2 \int_{\mathbb R^{n}}|\Psi(x')\nabla_{x''}\beta(x'')|^2\,dx'dx'' \\\nonumber
+&&\hskip-6mm C\varepsilon\,\hskip2mm \int_{\mathbb R^{n}} \mbox{ (all mixed terms) }\, dx'dx''.
\end{eqnarray}
Here ${\bf A}=(A^{hk}_{\alpha\beta})^{h,k=1,2}_{\alpha,\beta=1,2,\dots,m}$ and similarly ${\bf B}$ is the $(n-2)\times (n-2)$ minor in $h,k$.
The ``mixed terms" contain a single derivative falling on $\Psi$.
All the integrals are finite so we can let $\varepsilon \to 0+$ to obtain
\begin{eqnarray}&&\int_{\mathbb R^{n-2}}\beta^2(x'')\,dx''\,\times\\\nonumber
&&\, \Re e\int_{\mathbb R^2} \left\langle {\bf A}\left(\nabla_{x'} \Psi(x')-\left(1-\frac2p\right)\frac{\Psi(x')}{|\Psi|}\nabla_{x'}|\Psi|\right) , \nabla_{x'} \Psi(x') +\left(1-\frac2p\right)\frac{\Psi(x')}{|\Psi|}\nabla_{x'}|\Psi|\right\rangle dx'\\\nonumber
\ge&&C\int_{\mathbb R^{n-2}}\beta^2(x'')\,dx''\,\times \, \int_{\mathbb R^2} |\nabla_{x'}\Psi(x')|dx' .
\end{eqnarray}
Cancelling out the constant $\int_{\mathbb R^{n-2}}\beta^2(x'')\,dx''$ yields
\begin{eqnarray}\nonumber
&&\, \Re e\int_{\mathbb R^2} \left\langle {\bf A}\left(\nabla_{x'} \Psi(x')-\left(1-\frac2p\right)\frac{\Psi(x')}{|\Psi|}\nabla_{x'}|\Psi|\right) , \nabla_{x'} \Psi(x') +\left(1-\frac2p\right)\frac{\Psi(x')}{|\Psi|}\nabla_{x'}|\Psi|\right\rangle dx\\
\ge&&C \, \int_{\mathbb R^2} |\nabla_{x'}\Psi(x')|dx' .
\end{eqnarray}

The expression above has exactly the same form as \eqref{IC-mod} but now the dimension is two, as we have abandoned the variables $x_3,\dots,x_n$. Hence we just work with  \eqref{IC-mod} under the assumption that $n=2$. As in \cite[4.10]{CM14} we consider now test functions $\psi_R$ of the form
\begin{equation} 
\psi_R(x)=w(x)\eta(\log|x|/\log R),\quad\mbox{where}\quad w(x)=\mu\omega+\varphi(x),
\end{equation}
$|\omega|=1$, $\omega\in\mathbb C^m$, $\mu \in\mathbb R^+$, $R>1$ and $\eta\in C^\infty(\mathbb R)$ such that $\eta(t)=1$ when $t<1/2$ and $\eta(t)=0$ when $t\ge 1$. Assume that supp $\varphi\subset B_\delta(0)$ and that $|\nabla \phi| \in L^2$. Then, as follows from the calculation done in \cite[pp. 98-99]{CM14}, we have that
\begin{eqnarray}&&\nonumber
\lim_{R\to\infty}  \Re e\int_{\mathbb R^n} \left\langle A(x_0) \left(\nabla \psi_R-\left(1-\frac2p\right)\frac{\psi_R}{|\psi_R|}\nabla|\psi_R|\right) , \nabla \psi_R +\left(1-\frac2p\right)\frac{\psi_R}{|\psi_R|}\nabla|\psi_R|\right\rangle dx\\
&=&\Re e\int_{B_\delta(0)} \left\langle A(x_0) \left(\nabla w-\left(1-\frac2p\right)\frac{w}{|w|}\nabla|w|\right) , \nabla w +\left(1-\frac2p\right)\frac{w}{|w|}\nabla|w|\right\rangle dx\\
&\ge&C \int_{B_\delta(0)} |\nabla w|^2\,dx.\nonumber
\end{eqnarray}

The limiting argument above which allows us to pass from $\psi_R$ to $w$, as $R \to \infty$, relies on a calculation 
that works only in two dimensions; see the limit integral inequality on page 98 of \cite{CM14}.  We continue computing the above quantity
associated with this particular choice of $w$. 

Recall that $\nabla w=\nabla \varphi$. 
For the calculation of the other term involving $w$, we note that
\begin{equation} 
\frac{w}{|w|}\partial_k|w|=\frac{w}{|w|^2}\Re e\langle w,\partial_k w\rangle=|\mu\omega+\varphi|^{-2}(\mu\omega+\varphi)(\Re e\langle \mu\omega+\varphi,\partial_k\varphi\rangle.
\end{equation}
Observe that if we let $\mu\to\infty$ we see that the righthand side converges to $\omega(\Re e\langle\omega,\partial_k\varphi\rangle)$. It follows that
\begin{eqnarray}&&\nonumber
\lim_{\mu\to\infty}  \Re e\int_{B_\delta(0)} \left\langle A(x_0) \left(\nabla w-\left(1-\frac2p\right)\frac{w}{|w|}\nabla|w|\right) , \nabla w +\left(1-\frac2p\right)\frac{w}{|w|}\nabla|w|\right\rangle dx\\
&&=\Re e\int_{B_\delta(0)} \left\langle A^{hk}(x_0) \left(\partial_k \varphi-\left(1-\frac2p\right)\omega(\Re e\langle\omega,\partial_k\varphi\rangle)\right) , \partial_h \varphi +\left(1-\frac2p\right)\omega(\Re e \langle\omega,\partial_h\varphi\rangle)\right\rangle dx\nonumber\\
&&\ge C \int_{B_\delta(0)} |\nabla \varphi|^2\,dx.\label{eq"wrrom}
\end{eqnarray}

We now argue as in \cite[Theorem 2.6]{Y}. Let $\rho$ be a 2-periodic \lq\lq sawtooth" function that is equal to $t$ on $[0,1]$ and $2-t$ on $[1,2]$. Hence $\rho'(t)=\pm 1$ for a.e. $t\in\mathbb R$. For a fixed $\eta\in\mathbb C^m$ and $q=(q_1,0)\in\mathbb R^2$ consider test functions
$$\varphi_\varepsilon(x)=\varepsilon\zeta(x)\rho(q\cdot x/\varepsilon)\eta,$$
where $\zeta\in C_0^\infty(B_\delta(0))$ is real-valued. Observe that
$$\partial_k \varphi_\varepsilon(x)=\varepsilon \partial_k\zeta(x)\rho(q\cdot x/\varepsilon)\eta+\zeta(x)\rho'(q\cdot x/\varepsilon)q_k\eta.$$
Using this in \eqref{eq"wrrom} and letting $\varepsilon\to 0+$ we get
\begin{eqnarray}&&\nonumber
\lim_{\varepsilon\to0+}  \Re e\int_{B_\delta(0)} \left\langle A^{hk}(x_0) \left(\partial_k \varphi_\varepsilon-\left(1-\frac2p\right)\omega(\Re e\langle\omega,\partial_k\varphi_\varepsilon\rangle)\right) , \partial_h\varphi_\varepsilon +\left(1-\frac2p\right)\omega(\Re e \langle\omega,\partial_h\varphi_\varepsilon\rangle)\right\rangle dx\nonumber\\
&&=\Re e\int_{B_\delta(0)} \left\langle A^{hk}(x_0) \left(q_k\eta-\left(1-\frac2p\right)\omega(\Re e\langle\omega,q_k\eta\rangle)\right) , q_h\eta +\left(1-\frac2p\right)\omega (\Re e\langle\omega,q_h\eta\rangle)\right\rangle \zeta(x)^2\,dx\nonumber
\\
&&\ge C \int_{B_\delta(0)} |\eta|^2|q|^2\zeta(x)^2\,dx.\label{eq"wrrom3}
\end{eqnarray}
From this we conclude that in the special case $q =(q_1,0)$ we have proven \eqref{EllipLHp} 
because $\zeta$ is an arbitrary test function. That is:
\begin{equation}\label{EllipLHpsc}
\Re e\left\langle
(A^{11}(x)q_1^2)\left(\eta-\left(1-\frac2p\right)\omega(\Re e\langle\omega,\eta\rangle)  \right),\eta+\left(1-\frac2p\right)\omega(\Re e\langle\omega,\eta\rangle) \right\rangle \ge C |\eta|^{2}q_1^2,
\end{equation}
Now we treat the scalar $q_1$ as the first component of the vector $(q_1, 0,0, \dots, 0)$ in ${\mathbb R}^n$. For any $q\in{\mathbb R}^n$, choose an orthogonal matrix $R\in SO(n)$ such that $|q|e_1=Rq$. Then the change of variables 
$y=R^{T}z$ in \eqref{IC-mod} will transform this into an equivalent expression with the tensor $RAR^{T}$, where the indicated matrix multiplication for $A^{hk}_{\alpha\beta}$ occurs in the $h,k$ variables. We repeat the above calculation for this tensor giving us \eqref{EllipLHpsc} with $$|q|^2(RAR^{T})^{11}=
(|q|e_1)^TRAR^T(|q|e_1)=q^TAq=A^{hk}q_hq_k,$$
which gives \eqref{EllipLHp} in the general case.
\vglue2mm

\noindent $(iii)$. This is a trivial observation.\vglue2mm 

\noindent $(iv)$. We first realize that, since $(1-2/p)=-(1-2/p')$, taking adjoints shows that the $p$-ellipticity condition for $A$ implies that the $p'$-ellipticity holds for $A^*$.

If the condition \eqref{EllipLHp}  holds then, by choosing $\omega\bot\xi$, it follows that \eqref{EllipLH} must hold.\vglue2mm

\noindent $(v)$. If $A$ satisfies any one of the three $p$-ellipticity conditions for $p$ and for $q$, we first prove that it satisfies the same condition for any $r$ between $p$ and $q$.

We give the proof for the condition \eqref{EllipLp}, as the proofs for other two conditions are analogous.  Fix some $\xi\in \mathbb C^{n\times m}$ and $\omega\in\mathbb C^m$ with $|\omega|=1$. Let $\eta=\xi(\omega)$. Consider, for $x$ fixed, a function
$$f(t)=\Re e \left\langle  A(x)(\xi-t\eta), \xi+t\eta  \right\rangle_{\mathbb C^{n\times m}}=-t^2|\eta|^2+\dots,$$
for $t\in\mathbb R$. This is clearly a quadratic function in $t$ with negative leading coefficient, i.e., the graph of $f$ is a concave down parabola. Such a function attains its minimum on any bounded interval at its end-points. But the strong $p$ and $q$-ellipticity condition
on $A$ implies that
$f(1-2/p)\ge C|\xi|^2$ and $f(1-2/q)\ge C|\xi|^2$, for almost every $x \in \Omega$. It follows that $f(t)\ge C|\xi|^2$ for all $t$ between the points $1-2/p$ and $1-2/q$. Thus $A$ is strongly $r$-elliptic for all $r$ of the form $t=1-2/r$ with $t$ between the points $1-2/p$ and $1-2/q$.

The final claim is that $p$-ellipticity is an open property; that is, if $A$ is $p$-elliptic then for some small $\varepsilon=\varepsilon(A,p)>0$, $A$ is also $q$-elliptic for $|p-q|<\varepsilon$. Again, we check this assuming the condition \eqref{EllipLp} with the argument for the other conditions being similar. Using the notation we have introduced above we see that
for $\xi,\eta=\xi(\omega)\in \mathbb C^{n\times m}$ (here $|\eta|\le|\xi|$) we have for $t_0=1-2/p$

\begin{eqnarray}\nonumber
f(t)&=&f(t_0)+(t-t_0)\Re e \left\langle  A(x)(\xi-t_0\eta), \eta  \right\rangle-(t-t_0)\Re e \left\langle  A(x)\eta,\xi+t_0\eta  \right\rangle-\Re e(t-t_0)^2\left\langle  A(x)\eta,\eta \right\rangle\\
&\ge& C|\xi|^2-2|t-t_0|\|A\|_{L^\infty}|\xi\pm t_0\eta||\eta|-|t-t_0|^2\|A\|_{L^\infty}|\eta|^2\\\nonumber
&\ge& \left[C-\|A\|_{L^\infty}(4|t-t_0|+|t-t_0|^2)\right]|\xi|^2.
\end{eqnarray}
Hence for sufficiently small $|t-t_0|$ we can ensure that $f(t)\ge C|\xi|^2/2$. The strong $q$-ellipticity for $A$ follows for $q$ near $p$.

\end{proof}

\section{Regularity of solutions}

In this section we will study the regularity theory of solutions to the $L^p$-Dirichlet problem. These results consist of interior estimates and boundary estimates and will be used as a substitution of the De Giorgi-Nash-Moser regularity theory.

\subsection{Interior estimate}

We will prove the following interior estimate, which is similar to Theorem 1.1 of \cite{DP}. Our proof differs as we assume the integral $p$-ellipticity condition (which requires some changes in the argument). 

\begin{theorem}

Suppose $u \in W^{1,2}_{loc} (\Omega; \Bbb{C}^m)$ is the weak solution to the equation $\mathcal{L} u = \text{div} (A(x) \nabla u) + B(x) \cdot \nabla u = 0$ on $\Omega$. Let $p_0 = \sup \{ p > 1: A \text{ satisfies condition \eqref{eq:pellintp}} \}$ and assume that the coefficients of $B$ are in $L^\infty_{loc}(\Omega)$ and satisfy

\begin{equation}
  |B_{\alpha \beta}^h (x)| \leq \frac{K}{\delta(x)},
\end{equation}  
where $K$ is a uniform constant and $\delta(x)$ is the distance of $x$ to $\partial \Omega$. Then, for any $x \in \Omega$ and any
$r < \delta(x)/3$, if $B_r(x)$ is the ball centered at $x$ with radius $r$, we have:

\begin{itemize}
  \item (Caccioppoli-type inequality): 
    \begin{equation}
      \vint_{B_r(x)} |u|^{p-2} |\nabla u|^2  dy \leq \frac{C}{r^2} \vint_{B_{2r}(x)} |u|^p  dy,
    \end{equation}
    where $p \in (p_0', p_0)$ and the constant $C$ depends on $n,m,K,\Lambda$ and $\lambda_p$.
    
  \item (Reverse H\"older's inequality) For any $q \in [1, \frac{p_0 n}{n-2})$, 
  
  \begin{equation}
    \left( \vint_{B_r(x)} |u|^q dy \right)^{1/q} \leq C  \vint_{B_{2r} (x)} |u| dy ,
  \end{equation}
  where the constant $C$ depends on $n,m,K,\Lambda,$ and $\lambda_p$.
\end{itemize} 

\end{theorem}

\begin{proof}
  We first assume that coefficients of $A,B$ are all smooth, so that the classical solution $u$ is also smooth. We will later use an approximation argument similar to Section 7 of \cite{KP}. Let $\phi \in C_0^\infty (B_{2r}(x))$ satisfy $0 \leq \phi \leq 1$, $\phi \equiv 1$ on $B_r(x)$ and $|\nabla \phi| \lesssim \frac{1}{r}$. Set $v = u \phi^2$, we have
  
  \begin{equation}
    (\mathcal{L}v)^\alpha = u^\beta \partial_{h} ( A^{hk}_{\alpha \beta}(x) \partial_k (\phi^2) + B_{\alpha \beta} ^h \partial_h (\phi^2) )  + A_{\alpha \beta}^{hk}(x) \partial_h u^\beta \partial_k (\phi^2) + (A^*)_{\alpha \beta}^{hk}(x) \partial_h u^\beta \partial_k(\phi^2).
  \end{equation}   
  We multiply both sides by $|v|^{p-2}(\overline{v})^\alpha$, sum over $\alpha = 1,2, \dots, n$ and integrate, to get:
  
  \begin{eqnarray}\nonumber
    &&\int_\Omega  \langle A \nabla v , \nabla (|v|^{p-2} v) \rangle \\\nonumber
      &=& \int_\Omega \langle B \nabla v , |v|^{p-2} v \rangle + \int_\Omega A^{hk}\nabla_h \phi^2 \nabla_k (|v|^{p-2} u \overline{v}) - \int_\Omega B^h \nabla_h \phi^2 |v|^{p-2}u \overline{v} \\
      &+& \int_\Omega |v|^{p-2}\overline{v} \left( A^{hk}\nabla_h u \nabla_k \phi^2 + A^{*hk} \nabla_h u \nabla_k \phi^2 \right) \\
      &=: & I_1 + I_2 + I_3 + I_4.\nonumber
  \end{eqnarray}
  Here we use the inner product $\langle \,,\, \rangle$ to denote the summation over both upper and lower indices $h,k,\alpha, \beta.$ Also, for simplicity of notations, for each $h,k$ fixed, we write $A^{hk}$ as the $m \times m$ matrix $(A^{hk}_{\alpha \beta})_{\alpha , \beta = 1}^m$. Thus $I_1, I_2, I_3, I_4$ above are all scalar quantities.
   
  For $I_1$,
  \begin{eqnarray}\nonumber
    |I_1| &\lesssim & \frac{1}{r}\int_\Omega |v|^{p-1}|\nabla v| \\
    &\leq & \frac{1}{r} \left( \int_\Omega |v|^{p-2} |\nabla v|^2 \right)^{1/2} \left( \int_\Omega |v|^p \right)^{1/2} \\
    &\leq &\frac{1}{r} \left( \int_\Omega |v|^{p-2} |\nabla v|^2 \right)^{1/2} \left( \int_{B_{2r}(x)} |u|^p \right)^{1/2}.\nonumber
  \end{eqnarray}

  For $I_2$,
  
  \begin{eqnarray} \nonumber
   |I_2|
    &\lesssim & \left| \int_\Omega A^{hk}\phi \nabla_h \phi |v|^{p-3} \frac{\langle v,\nabla_k v \rangle}{|v|} u \overline{v} \right| + \left| \int_\Omega A^{hk}\phi \nabla_h \phi |v|^{p-2} \nabla_k u \overline{v} \right| \\
    &+& \left| \int_\Omega A^{hk}\phi \nabla_h \phi |v|^{p-2}u \nabla_k \overline{v} \right| \\
    &=:& I_{21} + I_{22} + I_{23}.\nonumber
  \end{eqnarray}
  For $I_{21}$,
  
  \begin{eqnarray}\nonumber
    I_{21} &\leq & \int_\Omega \phi |\nabla \phi| |v|^{p-2}|\nabla v| |u| \\
    &\lesssim & \left( \int_\Omega |v|^{p-2} |\nabla v|^2 \right)^{1/2} \left( \int_\Omega |v|^{p-2} |u|^2 \phi^2 |\nabla \phi|^2 \right)^{1/2} \\
    &\lesssim & \left( \int_\Omega |v|^{p-2} |\nabla v|^2 \right)^{1/2} \frac{1}{r} \left( \int_\Omega |u|^p \phi^{2p - 2} \right)^{1/2} \nonumber\\
    &\leq & \left( \int_\Omega |v|^{p-2} |\nabla v|^2 \right)^{1/2} \frac{1}{r} \left( \int_{B_{2r}(x)} |u|^p  \right)^{1/2},\nonumber
  \end{eqnarray}
For $I_{22}$, we will use the useful identity
  
  \begin{equation}
    \overline{v}\nabla u = \overline{u} \nabla v - |u|^2 \nabla \phi^2,
  \end{equation}
  then
  
  \begin{eqnarray}\nonumber
   I_{22} 
    &\lesssim & \int_\Omega \phi |\nabla \phi| |v|^{p-2} | \overline{u} \nabla v - |u|^2 \nabla \phi^2 | \\\nonumber
    &\lesssim & \int_\Omega \phi |\nabla \phi| |\nabla v| |v|^{p-2} |u| + \int_\Omega \phi^{2} |\nabla \phi| |v|^{p-2} |u|^2 |\nabla \phi| \\
    &\lesssim & \left( \int_\Omega |v|^{p-2} |\nabla v|^2 \right)^{1/2} \left( \int_\Omega \phi^{2} |\nabla \phi|^2 |v|^{p-2} |u|^2 \right)^{1/2} + \frac{1}{r^2} \int_\Omega |v|^{p-2} |u|^2 \phi^{2} \\
    &\lesssim & \left( \int_\Omega |v|^{p-2} |\nabla v|^2 \right)^{1/2} \frac{1}{r} \left( \int_\Omega \phi^{2} |v|^{p-2} |u|^2 \right)^{1/2} + \frac{1}{r^2} \int_\Omega |v|^{p-2} |u|^2 \phi^{2} \nonumber\\
    &\leq & \left( \int_\Omega |v|^{p-2} |\nabla v|^2 \right)^{1/2} \frac{1}{r} \left( \int_{B_{2r}(x)} |u|^p \right)^{1/2} + \frac{1}{r^2} \int_{B_{2r}(x)} |u|^p.\nonumber
  \end{eqnarray}
  For $I_{23}$,
  
  \begin{eqnarray}\nonumber
    I_{23} 
    &\lesssim & \int_\Omega \phi^{p-1} |\nabla \phi| |v|^{p-2}|\nabla v| |u| \\
    &\leq & \left( \int_\Omega |v|^{p-2} |\nabla v|^2 \right)^{1/2} \left( \int_\Omega |v|^{p-2} |u|^2 \phi^{2} |\nabla \phi|^2 \right)^{1/2} \\
    &\lesssim & \left( \int_\Omega |v|^{p-2} |\nabla v|^2 \right)^{1/2} \frac{1}{r} \left( \int_\Omega |u|^p \phi^{2p - 2} \right)^{1/2} \nonumber\\
    &\leq & \left( \int_\Omega |v|^{p-2} |\nabla v|^2 \right)^{1/2} \frac{1}{r} \left( \int_{B_{2r}(x)} |u|^p  \right)^{1/2}.\nonumber
  \end{eqnarray}
  Combine $I_{21}, I_{22}, I_{23}$, we get
  
  \begin{equation}
    |I_2| \lesssim \left( \int_\Omega |v|^{p-2} |\nabla v|^2 \right)^{1/2} \frac{1}{r} \left( \int_{B_{2r}(x)} |u|^p \right)^{1/2} + \frac{1}{r^2} \int_{B_{2r}(x)} |u|^p.
  \end{equation}
  For $I_3$,
  
  \begin{eqnarray} \nonumber
    |I_3| &\lesssim & \frac{1}{r^2} \int_\Omega \phi |v|^{p-1} |u| \\
    &\leq & \frac{1}{r^2} \int_{B_{2r}(x)} |u|^p,
  \end{eqnarray}
 For $I_4$, we have
  
  \begin{eqnarray}
    |I_4| \lesssim \int_\Omega \phi^{j-1} |\nabla \phi| |v|^{p-2} | \overline{u} \nabla v - |u|^2 \nabla \phi^j |.
  \end{eqnarray}
  Hence $I_4$ has same estimate as $I_{22}$. Combining the estimates of $I_1, I_2, I_3, I_4$ together with Theorem \ref{p-ellipticity lower bound}, and observing that $\phi \equiv 1$ on $B_r(x)$, we complete the proof of the Caccioppoli type inequality. Having established the Caccioppoli inequality, the following reverse H\"older's inequality:

\begin{equation}
  \left( \vint_{B_r(x)} |u|^q \right)^{1/q} \lesssim \left( \vint_{B_{2r}(x)} |u|^p \right)^{1/p}, \ p,q \in \left(p_0' , \frac{p_0 n}{n-2}\right)
\end{equation}  
follows directly from a use of the Poincar\'e-Sobolev inequality and an iteration argument. The proof of this part is exactly the same as the proof of either Theorem 1.1 in \cite{DP} or Lemma 3.1 in \cite{FMZ} so we omit the details here. 

We now establish an improvement of this reverse H\"older's inequality, using a convexity argument from Theorem 2.4 of \cite{S}. For any $0 < s < t < 1$, there exists a constant $c$ and a sequence of points $\{ x_l \}$ in $B_{tr}(x)$ such that $B_{2c(t-s)r}(x_l) \subset B_{tr}(x)$ so we have

\begin{eqnarray}\nonumber
  \vint_{B_{sr}(x)} |u|^q
    &\lesssim & \left( \frac{t-s}{s} \right)^n \sum_l \vint_{B_{c(t-s)r}(x_l) } |u|^q \\
    &\lesssim & \left( \frac{t-s}{s} \right)^n \sum_l \left( \vint_{B_{2c(t-s)r} (x_l)} |u|^2 \right)^{q/2} \\
    &\lesssim & \left( \frac{t-s}{s} \right)^n \left( \frac{t}{t-s} \right)^{nq/2} \left( \vint_{B_{tr}(x)} |u|^2 \right)^{q/2}.\nonumber
\end{eqnarray}
Since here $q \geq 2$, we let $\theta$ be such that $\frac{1}{2} = \frac{\theta}{q} + \theta$ so from H\"older's inequality,

\begin{eqnarray}
  \left( \vint_{B_{tr}(x)} |u|^2 \right)^{1/2}
    \leq \left( \vint_{B_{tr}(x)} |u|^q \right)^{\frac{1-\theta}{q}} \left( \vint_{B_{tr}(x)} |u| \right)^{\theta}.
\end{eqnarray}
Define 

\begin{equation}
  I(t) = \frac{\left( \vint_{B_{tr}(x)} |u|^q \right)^{1/q}}{\vint_{B_r (x)} |u|},
\end{equation}  
it suffices to show $I(1/2) \leq C$. From the above estimates, we have

\begin{equation}
  I(s) \lesssim s^{-n/q} t^{n(1/2 - \theta)} (t-s)^{n(\frac{1}{q} - \frac{1}{2})} (I(t))^{1- \theta},
\end{equation}
which implies

\begin{equation}
  \log I(t) \lesssim \log \left( s^{-n/q} t^{n(1/2 - \theta)} (t-s)^{n(\frac{1}{q} - \frac{1}{2})}  \right) + (1 - \theta) \log I(t).
\end{equation}
Choose $s = t^b$ with $b^{-1} > 1-  \theta$ and integrate both sides with respect to $\frac{dt}{t}$ from $1/2$ to $1$, we have

\begin{equation}
  \frac{1}{b} \int_{(1/2)^b}^1 \log I(t) \frac{dt}{t} \leq C + (1 - \theta) \int_{1/2}^1 \log I(t) \frac{dt}{t},
\end{equation}
which means

\begin{equation}
 (b^{-1} - \theta) I(1/2) \lesssim (b^{-1} - \theta)	 \int_{(1/2)^b}^1 \log I(t) \frac{dt}{t} \leq C.
\end{equation}
   To pass to the case when $A$ and $B$ are non-smooth, we let $A_l$ and $B_l$ be smooth approximations of $A$ and $B$ respectively. Let $u_l$ be the solution of $\mathcal{L}_l u_l = 0$, where $\mathcal{L}_l$ is the operator associated with the coefficients $A_l$ and $B_l$, with the same boundary condition as $u$. One can use the same argument as in \cite{DP} to show that $u_l \to u$ strongly in $W^{1,2}_{loc}(\Omega)$. Indeed, such an argument comes from \cite{KP} and it is easy to verify that Theorem 7.5 of \cite{KP} also holds for systems.

\end{proof}

\subsection{Boundary estimate}

We will show that both the Caccioppoli and reverse H\"older inequalities hold for solutions vanishing on an open subset of $\partial \Omega$ in the neighbourhood of such a set.  These estimates play an important role in extrapolation.

\begin{theorem}\label{Boundary estimate}
  Suppose $u \in  W^{1,2}_{loc} (\Omega; \Bbb{C}^m)$ is the weak solution to the equation $\mathcal{L} u = \text{div} (A(x) \nabla u) + B(x) \cdot \nabla u = 0$ on $\Omega$. Let $p_0 = \inf \{ p > 1: A \text{ satisfies condition \eqref{eq:pellintp} } \}$ and suppose the coefficients of $B$ are in $L^\infty_{loc}(\Omega)$ and satisfy

\begin{equation}
  |B_{\alpha \beta}^h (x)| \leq \frac{K}{\delta(x)},
\end{equation}  
where $K = K(n,m,p,\lambda_p)$ is a small enough constant and $\delta(x)$ is the distance of $x$ to $\partial \Omega$. Denote by $B_r(x)$ the ball centered at $x \in \partial \Omega$ with radius $r$. 
Suppose that for a given $x \in \Omega$, there exists $r>0$ such that $\text{Tr} (u) = 0$ on $B_{3r}(x) \cap \partial \Omega$. Then the following two inequalities hold:

\begin{itemize}
  \item (Caccioppoli-type inequality): For any $p \in (p_0', p_0)$ there exists $K>0$ and $C>0$ depending on $n,m,p,\lambda_p$ and $\|A\|_{L^\infty}$ such that
    \begin{equation}
      \vint_{B_r(x) \cap \Omega} |u|^{p-2} |\nabla u|^2  dy \leq \frac{C}{r^2} \vint_{B_{2r}(x) \cap \Omega} |u|^p  dy.
    \end{equation}
    
  \item (Reverse H\"older's inequality) For any $q \in [1, \frac{p_0n}{n-2})$, 
  
  \begin{equation}
    \left( \vint_{B_r(x) \cap \Omega } |u|^q dy \right)^{1/q} \leq C  \vint_{B_{2r} (x) \cap \Omega} |u| dy,
  \end{equation}
  where the constant $C$ depends on $n,m,K,\lambda_p$ and $\|A\|_{L^\infty}$.
\end{itemize} 
\end{theorem}

\begin{proof}

We first assume the coefficients of $A$ and $B$ are smooth. Let $\phi \in C_0^ \infty (B_{2r}(\Omega))$ satisfy $0 \leq \phi \leq 1$, $\phi \equiv 1$ on $B_r(x)$ and $|\nabla \phi| \lesssim \frac{1}{r}$. Set $v = u \phi^2$. We will employ a similar argument as used for the interior estimate. Since, at the boundary, we have $\delta (x) < r$, the only difference will be the estimate of those integrals involving $B$, i.e. terms $I_1$ and $I_3$. For $I_1$, we have

  \begin{eqnarray}\nonumber
    |I_1| &\leq & K \int_\Omega \frac{|\nabla v| |v|^{p-1} }{\delta(x)} \\\nonumber
    &\leq & K \left( \int_\Omega |v|^{p-2} |\nabla v|^2 \right)^{1/2} \left( \int_\Omega \frac{|v|^p}{\delta(x)^2} \right)^{1/2} \\
    &\leq & C(n,p) K \left( \int_\Omega |v|^{p-2} |\nabla v|^2 \right)^{1/2} \left( \int_\Omega |\nabla |v|^{\frac{p}{2}}|^2 \right)^{1/2} \\
    &\leq & C(n,m,p)K \int_\Omega |v|^{p-2} |\nabla v|^2,\nonumber
  \end{eqnarray}   
  where in the second to last line, we apply a one-dimensional Hardy's inequality in a direction transversal to the boundary
  to the function $|v|^{p/2}$.
  The term $I_1$ can be absorbed by $\lambda_p \int_\Omega |v|^{p-2}|\nabla v|^2$ if we choose $K < \frac{\lambda_p}{2 C(n,m,p)}$, since by the $p$-ellipticity,
  
\begin{equation}
  \lambda_p \int_\Omega |v|^{p-2}|\nabla v|^2 \leq \Re e\ \int_\Omega \langle A \nabla v, \nabla (|v|^{p-2} v) \rangle,
\end{equation}
and hence for small $K$ . We now estimate $I_3$:
  
  \begin{eqnarray}\nonumber
    |I_3| &\lesssim & \frac{K}{r} \int_\Omega \phi  |v|^{p-1} |u| \frac{1}{\delta(x)} \\
    &\leq & \frac{K}{r} \left( \int_\Omega \frac{|v|^p}{\delta(x)^2} \right)^{1/2} \left( \int_\Omega \phi^{2} |v|^{p-2} |u|^2 \right)^{1/2} \\
    &\lesssim & \frac{K}{r} \left( \int_\Omega |\nabla |v|^{\frac{p}{2}}|^2 \right)^{1/2} \left( \int_{\Omega} \phi^{2p-2} |u|^p \right)^{1/2} \nonumber\\
    &\lesssim & \frac{K}{r} \left( \int_\Omega |v|^{p-2} |\nabla v|^2 \right)^{1/2} \left( \int_{B_{2r}(x)} |u|^p \right)^{1/2},\nonumber
  \end{eqnarray}
  where for the first term of the second to last line, we apply the Hardy's inequality. Once we have Caccioppoli inequality, since we assume $\text{Tr} (u) = 0$ on $B_{3r}(x) \cap \partial \Omega$, we can still apply the Poincar\'e-Sobolev inequality at the boundary. Then we employ an iteration argument that is exactly the same as the one used in the interior estimate to obtain the reverse H\"older inequality.
\end{proof}

\section{$L^p$-Dirichlet problem and the extrapolation theory}\label{Dir}

This section is devoted to the extrapolation theory of the $L^p$-Dirichlet problem. We are going to show that given the solvability of the $L^p$-Dirichlet problem for some specific $p>1$, one can extrapolate such solvability to a larger exponent. To begin with, consider the operator $( \mathcal{L} u )_\alpha = \partial_h (A^{hk}_{\alpha \beta}(x) \partial_k u^\beta) + B^h_{\alpha \beta} \partial_h (x) u^\beta$ where the tensor $A (x) = (A^{hk}_{\alpha \beta}) (x)$ is elliptic and uniformly bounded and $B (x) = (B^h_{\alpha \beta} (x))$ satisfies $|B^h_{\alpha \beta} (x)| \leq \frac{K}{\delta(x)}$ with $K$ small enough (to be chosen later). For a Lipschitz domain $\Omega$, consider the sesquilinear form $\mathcal{B} : H^1 (\Omega; \Bbb{C}^m) \times H^1_0 (\Omega;\Bbb{C}^m) \to \Bbb{C}$ defined as following:

\begin{equation}\label{eq:sesq}
  \mathcal{B} [u,v] = \int_\Omega \langle A^{hk}_{\alpha \beta} \nabla_h u^\alpha , \nabla_k v^\beta  \rangle + \langle B^h_{\alpha \beta} \nabla_h u^\alpha , v^\beta \rangle dx, 
\end{equation} 
where $H^1 (\Omega; \Bbb{C}^m)$ is the $W^{1,2}$-Sobolev space equipped with the seminorm $||u||^2_{H^1} = \int_\Omega |\nabla u|^2 dx$ and $H^1_0 (\Omega;\Bbb{C}^m)$ is the closure of $C_0^\infty (\Omega; \Bbb{C}^m)$ with respect to the same seminorm. The sesquilinear form $\mathcal{B}$ is bounded from above since $A$ is uniformly bounded and the second term in the definition of $\mathcal{B}$ can be bounded using Hardy's inequality:

\begin{eqnarray}\nonumber
  \left\vert \int_\Omega  \langle B^h_{\alpha \beta} \nabla_h u^\alpha , v^\beta \rangle dx \right\vert 
    &\leq& K \left( \int_\Omega |\nabla u|^2 \right)^{1/2} \left( \int_\Omega \frac{|v|^2}{\delta (x)^2} \right)^{1/2}  \\
    &\leq& C(n)K\left( \int_\Omega |\nabla u|^2 \right)^{1/2} \left( \int_\Omega |\nabla v|^2 \right)^{1/2}.
\end{eqnarray}

On the other hand, throughout this paper we assume that the tensor $A$ satisfies the integral ellipticity condition, i.e.

\begin{equation}
 \Re e \int_\Omega \langle A^{hk}_{\alpha \beta} \nabla_h u^\alpha , \nabla_k u^\beta  \rangle \geq \lambda \int_\Omega |\nabla u|^2 \text{ for every } u \in H_0^1 (\Omega; \Bbb{C}^m),
\end{equation}
and therefore we have 

\begin{equation}
  | \mathcal{B}[u,u] | \geq (\lambda - K C(n)) \int_\Omega |\nabla u|^2,
\end{equation}
where $C(n)$ comes from the Hardy inequality. Thus if we further choose $K < \frac{\lambda}{C(n)}$, the sesquilinear form $\mathcal{B}$ is coercive. We denote $ \dot{B}^{2,2}_{1/2}(\partial \Omega; \Bbb{C}^m)$ as the trace of functions in $H^1(\Omega;\Bbb{C}^m)$, then for any $f \in \dot{B}^{2,2}_{1/2}(\partial \Omega; \Bbb{C}^m)$, there exists $v \in H^1(\Omega;\Bbb{C}^m)$ such that $\text{Tr}(v) = f$. Writing $u =u_0 + v$, we seek for $u_0 \in H^1_0 (\Omega;\Bbb{C}^m)$ such that 

\begin{equation}
  \mathcal{B}[u_0, w] = - \mathcal{B}[v,w] \text{ for any } w\in H^1_0(\Omega;\Bbb{C}^m).
\end{equation}
By the Lax-Milgram lemma, such a $u_0$ exists and is unique. This implies the existence and uniqueness of $u \in H^1(\Omega;\Bbb{C}^m)$ satisfying $\mathcal{L} u = 0$ and $\text{Tr} (u) = f $ and we call this $u$ the {\it energy solution}.

\medskip

To formulate the $L^p$-Dirichlet problem for elliptic systems, we need to introduce the $L^p$-averaged nontangential maximal function $\tilde{N}_{p,a}$. Unlike the real scalar equation, in our setting there is no De Giorgi-Nash-Moser regularity theory, and so the solution $u$ is not necessarily defined pointwise. For any Lipschitz domain $\Omega$ and $u \in L^p_{\text{loc}} (\Omega;\Bbb{C}^m)$, we define the $L^p$-averaged nontangential maximal function $\tilde{N}_{p,a}$ as 

\begin{equation}
  \tilde{N}_{p,a} (u) (Q) = \sup_{x \in \Gamma_a (Q)} u_p (x),
\end{equation}    
where $Q \in \partial \Omega$, $\Gamma_a (Q)$ is the standard non-tangetial cone at $Q$ with aperture parameter $a$ and 

\begin{equation}
  u_p(x) = \left( \vint_{B_{\delta(x)/2} (x)} |u(y)|^p dy \right)^{1/p}.
\end{equation}
The $L^p$-Dirichlet problem can then be formulated as follows:

\begin{definition}
  Let $\Omega$ be a bounded Lipschitz domain, or alternatively, a domain above a Lipschitz graph. Consider the following Dirichlet problem
  
  \begin{equation}
    \begin{cases}
      \mathcal{L} (u) = 0 &\text{ on } \Omega, \\
      u(Q) = f(Q) \ \ \ &\text{ for } \sigma \text{-a.e.} \ Q \in \partial \Omega, \\
      \tilde{N}_{2,a}(u)(Q) \in L^p (\partial \Omega). 
    \end{cases}\label{eq-12}
  \end{equation}
  Given any $1 < p < \infty$, the $L^p$-Dirichlet problem is solvable if for any $f \in L^p (\partial \Omega; \Bbb{C}^m) \cap \dot{B}^{2,2}_{1/2}(\partial \Omega; \Bbb{C}^m)$, the unique energy solution $u$ satisfies
  
  \begin{equation}\label{eq-13}
    ||\tilde{N}_{2,a}(u)||_{L^p (\partial \Omega;d \sigma)} \leq C ||f||_{L^p (\partial \Omega;d \sigma)}
  \end{equation}
  for some constant $C = C(p,\Omega)$.
\end{definition}

\begin{remark}
The above $L^p$-Dirichlet problem associated with the adapted nontangential maximal function was introduced in \cite{DHM,DP}. 
Suppose that
$A$ satisfies \eqref{EllipIC} and let
$$p_0 = \sup \{ p: A \ \text{satisfies condition  \eqref{eq:pellintp} }  \}.$$ It has been shown in \cite{DP,DP2} that for different apertures $a,a'$ in some proper range and $p,q \in [1 , \frac{p_0 n}{n-2})$, $\tilde{N}_{p,a}(u)$ and $\tilde{N}_{q,a'}(u)$ are comparable in $L^r (\partial \Omega)$ for any $r > 0$. Hence,  in \eqref{eq-12} there is no difference between considering $\tilde{N}_{2,a}(u)$ and considering $\tilde{N}_{q,a'}(u)$ in the stated range of $q$ and $a'$.

On the other hand, since $L^p (\partial \Omega; \Bbb{C}^m) \cap \dot{B}^{2,2}_{1/2}(\partial \Omega; \Bbb{C}^m)$ is dense in $L^p (\partial \Omega; \Bbb{C}^m)$, by using a density argument and the estimate \eqref{eq-13} it follows that the solution operator $f\mapsto u$ can be extended to the whole space  $L^p (\partial \Omega; \Bbb{C}^m)$.

Moreover, under such extension, the solution $u$ attains its boundary data $f \in L^p (\partial \Omega; \Bbb{C}^m)$ in the following sense

\begin{equation}
  f(Q) = \lim_{x \in \Gamma_a(Q), x \to Q} \vint_{B_{\delta(x)/2} (x)} u(y) dy \qquad \text{for } \sigma{-a.e. }\, Q \in \partial \Omega.
\end{equation}
\end{remark}

\medskip

With these definitions in hand, we can now begin the proof of Theorem \ref{Extrapolation}.

\begin{proof}
With the help of the interior and boundary estimates established in the previous section, the proof of the above theorem is almost identical to the one in \cite{DP2} and we will just sketch it here. Theorem \ref{Extrapolation} relies on a real variable argument developed in \cite{S1} and in our setting, it suffices to prove the following reverse H\"older inequality:

\begin{equation}\label{reverse holder of N}
  \left( \vint_{B_r (Q) \cap \partial\Omega} |\tilde{N}_{2,a}(u)|^q d\sigma\right)^{1/q} \leq C \left( \vint_{B_{2r}(Q) \cap \partial\Omega} |\tilde{N}_{2,a}(u)|^p d\sigma \right)^{1/p},
\end{equation}
where $1 < p \leq q < \frac{p_0(n-1)}{n-2}$, $Q \in \partial 
\Omega$ and the energy solution $u$ vanishes on $B_{4r}(Q) \cap \partial\Omega$. Define 

\begin{eqnarray}\nonumber
  &&\mathcal{M}_1(u)(Q) = \sup_{x \in \Gamma_a(Q), \delta(x) \leq r} u_2(x), \\
  &&\mathcal{M}_2(u)(Q) = \sup_{x \in \Gamma_a(Q), \delta(x) > r} u_2(x),
\end{eqnarray}  
where $u_2 (x) = \left( \vint_{B_{\delta(x)/2} (x)} |u (y)|^2 dy \right)^{1/2}$. After some careful geometric observations, one obtains

\begin{equation}
  \mathcal{M}_2(u)(Q) \leq C \left( \vint_{B_{2r}(Q) \cap \partial\Omega} |\tilde{N}_{2,a} (u) (P)|^p  d\sigma(P) \right)^{1/p}.
\end{equation}
To estimate $\mathcal{M}_1(u)(Q)$, for any $x \in \Gamma_a(Q), \delta(x) \leq r$, we construct a sequence of balls $B_{r_j}(x_j)$ in $\Gamma_a(Q)$ with appropriate scale, such that $x_0 = x$ and $x_j \to Q$. Let $v(x) = |u|^{\frac{p}{2}-1} u(x)$, we estimate each $\vint_{B_{r_j}(x_j)} |v|^2 - \vint_{B_{r_{j+1}}(x_{j+1})} |v|^2$ by making use of the Poincar\'e inequality and then summing. Due to the fact that $u$ vanishes on $B_{4r}(Q) \cap \partial\Omega$, one can verify the following pointwise estimate:

\begin{equation}
  \mathcal{M}_1(v_2)(Q) \leq C r^{1/2} \int_{B_{2r}(Q) \cap \partial\Omega} \frac{A(\nabla v)(P)}{|P - Q|^{n - 3/2} }d\sigma(P),
\end{equation}
where $v_2$ denotes the $L^2$-average of $v$ over $B_{\delta(x)/2} (x)$ and 

\begin{equation}
  A^2(\nabla v)(P) = r^{-1} \int_{\Gamma_{\tilde{a}}^{2r} (P)} \frac{|\nabla v (y)|^2}{\delta(y)^{n-1}} dy,
\end{equation}
where $\Gamma_{\tilde{a}}^{2r}$ is the truncated cone with height $2r$ and appropriate aperture parameter $\tilde{a}$. With this pointwise estimate of $\mathcal{M}_1$ in hand, we apply the Sobolev inequality to the fractional integral and combine it with the boundary estimate (Theorem \ref{Boundary estimate}). The desired estimate follows after we recover back $u$. For further details see \cite{DP2}.
\end{proof}

\begin{remark}
  In fact, the assumptions on domain $\Omega$ can be further relaxed; $\Omega$ can be assumed to be a chord-arc domain. See \cite{DP2} for the necessary modifications in such setting. 
\end{remark}

\section{Application to the Lam\'e system}\label{five}

In this section, we show how the results of the previous section apply to a specific elliptic system, namely the Lam\'e system of linear elasticity. Recall that the variable coefficient Lam\'e system is given by

\begin{eqnarray}\label{eq-Lame-2}
  \mathcal{L} u = \nabla \cdot (\lambda (x) (\nabla \cdot u) I + \mu(x) (\nabla u + (\nabla u)^T)),
\end{eqnarray}
where $u: \Omega \to \Bbb{R}^n$ is a vector-valued function and the bounded measurable functions $\lambda,\mu$ are so-called Lam\'e coefficients describing elastic properties of the material at a given point. If we write \eqref{eq-Lame-2} in the form \eqref{eq:dopS} we see that the corresponding coefficients of $A^{hk}_{\alpha \beta}(x)$, with
$m=n$ throughout, are these:

\begin{eqnarray}\label{eq-Lame-coeff}
  &&A^{hk}_{\alpha \beta}(x) = \mu(x) \delta^{hk} \delta_{\alpha \beta} + \lambda(x) \delta^h_\alpha \delta^k_\beta + \mu(x) \delta^h_\beta \delta^k_\alpha, \qquad B^h_{\alpha \beta} (x) = 0.
\end{eqnarray}
 Here and throughout this section we assume that $\Omega$ is a bounded Lipschitz domain, or alternatively, an unbounded domain in $\mathbb R^n$ above a Lipschitz graph.
 
In \cite[Corollary 1.2]{DHM}, it has been shown that under certain assumptions on the Lam\'e coefficients $\mu(x)$ and $\lambda(x)$, the corresponding $L^p$-Dirichlet problem is solvable for $p \in (2 - \epsilon , \frac{2(n-1)}{n-2}+\epsilon)$, where $\epsilon>0$ is a small positive constant. We assume the same conditions on the ${Lip}_{loc}(\Omega)$ functions $\mu(x)$ and $\lambda(x)$ here. The proof of solvability of the Dirichlet problem relies heavily on the fact that the equation \eqref{eq-Lame-2} can be written as \eqref{eq:dopS} in multiple ways. To be precise, for an auxiliary function $r(x) \in \mbox{Lip}_{loc}(\Omega)$,  equation \eqref{eq-Lame-2} can be also written as
  
\begin{equation}\label{Lame}
  \mathcal{L}(u) = \partial_h ( {A}^{hk}_{\alpha \beta}(r) \partial_k u^\beta ) + {B}^h_{\alpha \beta}(r) \partial_h u^\beta,
\end{equation} 
where

\begin{eqnarray}\label{eq-Lame-coeff-m}
  &&{A}^{hk}_{\alpha \beta}(r) (x) = \mu(x) \delta^{hk} \delta_{\alpha \beta} + ( \lambda(x) + r(x) )\delta^h_\alpha \delta^k_\beta + ( \mu(x) - r(x) ) \delta^h_\beta \delta^k_\alpha, \\\nonumber
  && {B}^h_{\alpha \beta}(r) (x) = \partial_k r(x) (\delta^h_\alpha \delta^k_\beta - \delta^h_\beta \delta^k_\alpha  ).
\end{eqnarray}

The coefficients of \eqref{eq-Lame-coeff-m} enjoy certain special properties that one cannot assume for general systems. First, the Lam\'e coefficients are real, and hence so are the coefficients of  \eqref{eq-Lame-coeff-m} as long as $r$ is chosen to be real. Second, we will take advantage of the symmetry:  ${A}^{hk}_{\alpha \beta}(r)={A}^{kh}_{\beta\alpha}(r)$. As we shall see below, this will simplify the $p$-ellipticity condition that we will impose on the equation.

Finally, we note that when $p=2$, i.e., when we consider the usual ellipticity conditions, it is the case that ${A}(r)$ satisfies the Legendre-Hadamard property when $\mu > 0 $ and $\lambda > -2\mu$ \cite{BM09}. Moreover, ${A}(r)$ is guaranteed to be strongly elliptic if, in addition, it is the case that $0 < r < 2\mu$ and 
$\lambda+2\mu-r\ge 0$.  Strong ellipticity may also hold even when $\lambda+2\mu-r\le 0$ depending on the relationship between
$r$ and $\mu$.  See (1.16) of  \cite{BM09} for these facts.

\subsection{Sufficient conditions.}

Our aim is to figure out when the system  \eqref{Lame}  satisfies the integral $p$-ellipticity condition.  It is not easy to check
the validity of an integral condition. Instead, we resort to checking the strong $p$-ellipticity condition \eqref{EllipLp}. Because $A(r)$ is real and ${A}^{hk}_{\alpha \beta}(r)={A}^{kh}_{\beta\alpha}(r)$ the cross-terms (containing both $\xi$ and $\xi(\omega)$) cancel out.  Thus, we are left with

\begin{eqnarray}\label{EllipLp2}
&&\inf_{|\omega|=1} \left[\left\langle  {A}(r)(x)\xi,\xi  \right\rangle_{\mathbb R^{n\times n}} -\left(1-\frac2p\right)^2\left\langle  {A}(r)(x)\xi(\omega),\xi(\omega)  \right\rangle_{\mathbb R^{n\times n}} \right]
\ge C|\xi|^2.
\end{eqnarray}

Observe also that the expression in the brackets in \eqref{EllipLp2} is simply
\begin{equation}
   \mu |\xi|^2 + (\lambda + r) |\text{Tr } \xi|^2 + (\mu - r) \xi \circ \xi  - \left( 1 - \frac{2}{p} \right)^2 \left( \mu |q|^2 + (\lambda +\mu)\langle \omega, q\rangle^2 \right),
\end{equation}
where $\text{Tr}\xi = \xi_1^1 + ... + \xi_n^n$,  $\xi \circ \xi = \sum_{h,k = 1}^n \xi_h^k \xi_k^h$ and $q=q(\omega)=(q_1,q_2,\dots,q_n)\in\mathbb R^n$
is defined by $q_h=\langle \omega,\xi_h\rangle$.
Hence the condition \eqref{EllipLp2} will be true in the largest possible interval of $p$'s if the following inequality is satisfied:
\begin{eqnarray}\label{EllipLp4}
   \left( 1 - \frac{2}{p} \right)^2 
     &<&\mbox{ess} \inf_{x\in\Omega}\sup_{r \in\mathbb R} \inf_{|\omega|=1} \frac{ \mu(x) |\xi|^2 + (\lambda(x) + r) |\text{Tr } \xi|^2 + (\mu(x) - r) \xi \circ \xi }{ \mu(x) |q|^2 + (\lambda(x) +  \mu(x))\langle \omega, q\rangle^2  }.
\end{eqnarray}

Since the ball $|\omega|=1$ in $\mathbb C^n$ is compact it follows that for a fixed $x\in\Omega$, $\xi \in \Bbb{C}^{n \times n}$ and $r\in\mathbb R$ the infimum over $\omega$ of the expression in \eqref{EllipLp4} is attained. Thus, the $r$ depends on $x$ but we suppress this dependence in the
notation that follows. Let $\omega$ be such minimiser. We claim that without any loss of generality we can assume that $\omega=(1,0,0,\dots,0)^T=e_1$. Indeed, let $R$ be an orthogonal transformation, i.e., an element of SO$(n)$) such that $e_1=R\omega$. Let $\eta=R\xi R^T$. We claim that $(e_1,\eta)$ is also a minimiser of \eqref{EllipLp4}.

To see this, consider first the denominator of the expression \eqref{EllipLp4}. Denote by $\tilde{q}_h=\langle e_1,\eta_h\rangle$. Clearly, 
$$\tilde q:=\eta^T e_1=(R\xi R^T)^TR\omega=R\xi^T\omega=Rq.$$
By orthogonality of $R$ then clearly $|\tilde{q}|=|q|$ and $\langle \omega, q\rangle=\langle R\omega, Rq\rangle=\langle e_1, \tilde{q}\rangle$. It follows that the value of the denominator of \eqref{EllipLp4} for $(\omega,\xi)$ and $(e_1,\eta)$ is the same.  

Similarly, in the numerator we have $|\xi|=|\eta|$, Tr $\xi=$Tr $\eta$ as the matrices $\xi$ and $\eta$ are similar and hence have the same set of eigenvalues. Finally,
$$\eta\circ\eta=\eta_h^k \eta_k^h=R^k_i\xi_i^jR^h_j R^h_m\xi_m^nR^k_n.$$
Because $R$ is orthogonal $R^h_j R^h_m=\delta_{jm}$ and $R^k_iR^k_n=\delta_{in}$. Therefore
$$\eta\circ\eta=R^k_i\xi_i^jR^h_j R^h_m\xi_m^nR^k_n=\xi_i^j\xi_j^i=\xi\circ\xi.$$
It follows that  the value of the numerator of \eqref{EllipLp4} for $(\omega,\xi)$ and $(e_1,\eta)$ is the same as well. Given this the inequality we want to be satisfied is:
\begin{eqnarray}\label{EllipLp5}
   \left( 1 - \frac{2}{p} \right)^2 
     &<&\mbox{ess} \inf_{x\in\Omega}\sup_{r \in\mathbb R} \inf_{|\xi|=1} \frac{ \mu(x) |\xi|^2 + (\lambda(x) + r) |\text{Tr } \xi|^2 + (\mu(x) - r) \xi \circ \xi }{ \mu(x) ((\xi_1^1)^2+(\xi_2^1)^2+\dots+(\xi_n^1)^2) + (\lambda(x) +  \mu(x))(\xi_1^1)^2  }.
\end{eqnarray}

To analyse \eqref{EllipLp5} further let us drop the dependence of $\lambda,\,\mu$ on $x\in\Omega$ by fixing the point $x$. Let us relabel the parameter $r$ so that $\gamma=\mu-r$ and hence $\lambda+r=\lambda+\mu-\gamma$. In this notation we have
\begin{eqnarray}\label{EllipLp6}
   \left( 1 - \frac{2}{p} \right)^2 
     &<&\sup_{\gamma \in\mathbb R} \inf_{|\xi|=1} \frac{ \mu |\xi|^2 + (\lambda+\mu -\gamma) |\text{Tr } \xi|^2 + \gamma (\xi \circ \xi) }{ \mu ((\xi_2^1)^2+\dots+(\xi_n^1)^2) + (\lambda +  2\mu)(\xi_1^1)^2  }.
\end{eqnarray}
Clearly for the last term of the numerator we have
$$\gamma (\xi \circ \xi)\ge \gamma((\xi_1^1)^2+\dots+(\xi_n^n)^2)-|\gamma|\sum_{h\ne k}|\xi_h^k\xi_k^h|.$$
We further split the second sum
$$|\gamma|\sum_{h\ne k}|\xi_h^k\xi_k^h|\le 2|\gamma|\sum_{k=2}^n|\xi_1^k\xi_k^1|+2|\gamma|\sum_{2\le h<k}|\xi_h^k\xi_k^h|\le
2|\gamma|\sum_{k=2}^n|\xi_1^k\xi_k^1|+|\gamma|\sum_{h,k=2}^n(\xi_h^k)^2,$$
where we have used the AG-inequality. We use the same inequality for the penultimate term as well but with uneven weights which will give us
$$|\gamma|\sum_{h\ne k}|\xi_h^k\xi_k^h|\le \mu \sum_{k=2}^n(\xi_1^k)^2 +\frac{\gamma^2}{\mu}\sum_{k=2}^n(\xi_k^1)^2+|\gamma|\sum_{h,k>1,h\ne k}(\xi_h^k)^2.$$
Hence
$$\gamma (\xi \circ \xi)\ge \gamma((\xi_1^1)^2+\dots+(\xi_n^n)^2)-\mu \sum_{k=2}^n(\xi_1^k)^2 -\frac{\gamma^2}{\mu}\sum_{k=2}^n(\xi_k^1)^2-|\gamma|\sum_{h,k>1,h\ne k}(\xi_h^k)^2.$$
It follows that
\begin{eqnarray}\label{EllipLp7}
 && \frac{ \mu |\xi|^2 + (\lambda+\mu -\gamma) |\text{Tr } \xi|^2 + \gamma (\xi \circ \xi) }{ \mu ((\xi_2^1)^2+\dots+(\xi_n^1)^2) + (\lambda +  2\mu)(\xi_1^1)^2  }\ge\\\nonumber
&&\frac{(\mu-|\gamma|)\sum_{h,k>1,h\ne k}(\xi_h^k)^2+(\mu-\frac{\gamma^2}{\mu})\sum_{k=2}^n(\xi_k^1)^2+(\mu+\gamma)\sum_h(\xi_h^h)^2+ (\lambda+\mu -\gamma) |\text{Tr } \xi|^2}{ \mu \sum_{k=2}^n(\xi_k^1)^2 + (\lambda +  2\mu)(\xi_1^1)^2  }.
\end{eqnarray}
Consider first the off-diagonal elements. Observe that the terms $(\xi_1^k)^2$ have been eliminated.
The remaining off-diagonal terms make an appearance in  the first and second term in the numerator and the first term of the denominator of \eqref{EllipLp7}.

Under the assumption that $|\gamma|\le \mu$ (and hence $\mu-|\gamma|,\mu-\frac{\gamma^2}\mu\ge 0$) each remaining off-diagonal element $x=(\xi_h^k)^2$ appears in \eqref{EllipLp7} in the form of a rational function
$x\mapsto \frac{Ax+B}{Cx+D}$ when $A,C\ge0,\,D>0$, which is monotone for nonnegative $x$. Hence, it attains its supremum on $[0,\infty)$ either when $x=0$ (where the value is $\frac{A}C$) or in the limit $x\to\infty$ (where the value is $\frac{B}D$). Applying this observation for each remaining off-diagonal element $(\xi_h^k)^2$ we conclude that
\begin{eqnarray}\label{EllipLp8}
 && \frac{ \mu |\xi|^2 + (\lambda+\mu -\gamma) |\text{Tr } \xi|^2 + \gamma (\xi \circ \xi) }{ \mu ((\xi_2^1)^2+\dots+(\xi_n^1)^2) + (\lambda +  2\mu)(\xi_1^1)^2  }\ge\\\nonumber
&&\qquad\qquad  \min\left\{1-\frac{\gamma^2}{\mu^2},
\frac{(\mu+\gamma)\sum_h(\xi_h^h)^2+ (\lambda+\mu -\gamma) |\text{Tr } \xi|^2}{  (\lambda +  2\mu)(\xi_1^1)^2  }\right\}.
\end{eqnarray}

Clearly, \eqref{EllipLp8} indicates that we only want to choose $\gamma$ such that $|\gamma|<\mu$ so that the right-hand side is positive. This also implies that $\mu+\gamma>0$. The coefficient $\mu+\lambda+\gamma$ can be either positive or negative.\vglue2mm

Let $A=\sum_{h>1}\xi_h^h$. Then
$$|\text{Tr } \xi|^2=(\xi_1^1)^2+2\xi_1^1A+A^2,\quad\mbox{and}\quad
A^2\le (n-1)\sum_{h>1}(\xi_h^h)^2.
$$
It follows that
$$(\mu+\gamma)\sum_h(\xi_h^h)^2+ (\lambda+\mu -\gamma) |\text{Tr } \xi|^2\ge (\lambda+2\mu) (\xi_1^1)^2+2(\lambda+\mu-\gamma)\xi_1^1A+\left(\lambda+\mu-\gamma+\frac{\mu+\gamma}{n-1}\right)A^2.$$

This term must be positive and therefore it implies another condition $\gamma$ must satisfy. 
\begin{equation}\label{eq-cgamma}
(\lambda+2\mu)(\lambda+\mu-\gamma+\frac{\mu+\gamma}{n-1})>(\lambda+\mu-\gamma)^2,\quad\mbox{and}\quad \gamma<\frac{n-1}{n-2}\lambda+\frac{n}{n-2}\mu.
\end{equation}
Given this
\begin{equation}\label{eq-77}
(\mu+\gamma)\sum_h(\xi_h^h)^2+ (\lambda+\mu -\gamma) |\text{Tr } \xi|^2\ge\left[ \lambda+2\mu-\frac{(\lambda+\mu-\gamma)^2}{\lambda+\mu-\gamma+\frac{\mu+\gamma}{n-1}}\right](\xi_1^1)^2.
\end{equation}
Hence by \eqref{EllipLp8} we have

\begin{eqnarray}\label{EllipLp9}
 && \frac{ \mu |\xi|^2 + (\lambda+\mu -\gamma) |\text{Tr } \xi|^2 + \gamma (\xi \circ \xi) }{ \mu ((\xi_2^1)^2+\dots+(\xi_n^1)^2) + (\lambda +  2\mu)(\xi_1^1)^2  }\ge\\\nonumber
&&\qquad\qquad  \min\left\{1-\frac{\gamma^2}{\mu^2},1-\frac{(\lambda+\mu-\gamma)^2}{(\lambda+2\mu)(\lambda+\mu-\gamma+\frac{\mu+\gamma}{n-1})
 }\right\}.
\end{eqnarray}
Notice that there is a tension between the two terms inside the minimum. The first one is largest at $\gamma=0$, while the second one at $\gamma=\lambda+\mu$ with the largest value $1$. However, in the special case when $\mu=-\lambda$ both attain its maximum at $\gamma=0$.

From this we first derive a dimension independent bound. Given the restriction $|\gamma|<\mu$ we see that
\begin{equation}\label{eq-78}
1-\frac{(\lambda+\mu-\gamma)^2}{(\lambda+2\mu)(\lambda+\mu-\gamma+\frac{\mu+\gamma}{n-1})}
> 1-\frac{\lambda+\mu-\gamma}{\lambda+2\mu}=\frac{\mu+\gamma}{\lambda+2\mu}>0.
\end{equation}
This is only correct when $\lambda+\mu-\gamma\ge 0$. Thus we consider two cases.\vglue2mm

\noindent {\bf Case 1: $\lambda+\mu <0$.} Here the conditions we have for $\gamma$ imply that $-\mu<\gamma\le \lambda+\mu<0$. On this interval the function $1-\frac{\gamma^2}{\mu^2}$
attains its maximum at the endpoint $\gamma=\lambda+\mu$. But this endpoint is also the maximum of
the first term of \eqref{eq-78} with value $1$. In reality we might do slightly better because, in \eqref{eq-78}, we have neglected a small positive term in the denominator and so a choice $\gamma=\lambda+\mu+\varepsilon$ will work better for \eqref{EllipLp9}.

This leads us to conclude that
in this case if
\begin{eqnarray}\label{EllipLp10a}
   \left( 1 - \frac{2}{p} \right)^2 
     &<&1-\left(\frac{\lambda+\mu}{\mu+\varepsilon}\right)^2,
\end{eqnarray}
we have that \eqref{EllipLp2} holds for $r\approx-\lambda$ and some small $\varepsilon>0$.\vglue2mm

\noindent {\bf Case 2: $\lambda+\mu \ge 0$.}  Here $\gamma$ can take values between
$-\mu$ and $\lambda+\mu$ (this interval contains zero). We observe that the function on the righthand side of \eqref{eq-78} is monotone increasing in $\gamma$. Also $1-\frac{\gamma^2}{\mu^2}$ attains the same value at $\gamma$ and $-\gamma$. These two facts imply that to find an optimal value of $\gamma$  maximising 
$$\min\left\{\frac{\mu+\gamma}{\lambda+2\mu},1-\frac{\gamma^2}{\mu^2}\right\}$$
it suffices to look for $\gamma$ in the interval $[0,\lambda+\mu]$. Since one function increases in this interval and the other one is decreasing there will be one point in the interval where the values of these two functions are equal. At such point:

$$1-\frac{\gamma^2}{\mu^2}=\frac{(\mu-\gamma)(\mu+\gamma)}{\mu^2}=\frac{\mu+\gamma}{\lambda+2\mu}\quad\Rightarrow\quad \gamma=\mu-\frac{\mu^2}{\lambda+2\mu}.$$
Hence
$$1-\frac{\gamma^2}{\mu^2}=\frac{\mu+\gamma}{\lambda+2\mu}=\frac{2\mu}{\lambda+2\mu}
-\frac{\mu^2}{(\lambda+2\mu)^2}=\frac{\mu(2\lambda+3\mu)}{(\lambda+2\mu)^2}=\frac{(\lambda+2\mu)^2-(\lambda+\mu)^2}{(\lambda+2\mu)^2}=1-\left(\frac{\lambda+\mu}{\lambda+2\mu}\right)^2.$$

Again by the same consideration because of the sharp inequality in \eqref{eq-78} there might
be a small improvement in what we've just calculated. It follows that for $\lambda$ and $\mu$ such that
\begin{eqnarray}\label{EllipLp10b}
   \left( 1 - \frac{2}{p} \right)^2 
     &<&1-\left(\frac{\lambda+\mu}{\lambda+2\mu +\varepsilon}\right)^2,
\end{eqnarray}
we have that \eqref{EllipLp2} holds for $r\approx\frac{\mu^2}{\lambda+2\mu}$ and some small $\varepsilon>0$.\vglue2mm

Merging these two cases we see that there exists $r\in\mathbb R$ and some small $\varepsilon>0$ such that  \eqref{EllipLp2}
holds for all $p$ given by the inequality
\begin{eqnarray}\label{EllipLp10}
   \left( 1 - \frac{2}{p} \right)^2 
     &<&1-\left(\frac{\lambda+\mu}{\max\{\mu,\lambda+2\mu\} +\varepsilon}\right)^2,
\end{eqnarray}
\vglue2mm

Another interesting special case is when $n=2$. In this case,
$$1-\frac{(\lambda+\mu-\gamma)^2}{(\lambda+2\mu)(\lambda+\mu-\gamma+\mu+\gamma)}
= 1-\frac{(\lambda+\mu-\gamma)^2}{(\lambda+2\mu)^2}.
$$
Again, we find $\gamma$ for which this term equals to $1-\frac{\gamma^2}{\mu^2}$, and find that the optimal $\gamma$ is
$$\gamma=\frac{\mu(\lambda+\mu)}{\lambda+3\mu},\quad\Rightarrow\quad 1-\frac{(\lambda+\mu-\gamma)^2}{(\lambda+2\mu)^2}=1-\frac{\gamma^2}{\mu^2}=1-\left(\frac{\lambda+\mu}{\lambda+3\mu}\right)^2.$$
It follows that when $n=2$ for $\lambda$ and $\mu$ such that
\begin{eqnarray}\label{EllipLp11}
   \left( 1 - \frac{2}{p} \right)^2 
     &<&1-\left(\frac{\lambda+\mu}{\lambda+3\mu}\right)^2,
\end{eqnarray}
we have that \eqref{EllipLp2} holds for $r=\frac{2\mu^2}{\lambda+3\mu}$. Observe that since
$\lambda+3\mu\ge \max\{\mu,\lambda+2\mu\}$ we see that \eqref{EllipLp11} is always an improvement of the estimate \eqref{EllipLp10}.

Finally, all of the above calculations can be made precise in all dimensions, with dimension-dependent bounds on $(1-2/p)^2$. We start by setting 

\begin{equation}
  1-\frac{\gamma^2}{\mu^2}=1-\frac{(\lambda+\mu-\gamma)^2}{(\lambda+2\mu)(\lambda+\mu-\gamma+\frac{\mu+\gamma}{n-1})
 }.
\end{equation}
For simplicity, if we rewrite $x = \frac{\gamma}{\mu}$ and $a = \frac{\lambda}{\mu}$, then we obtain the following equation

\begin{equation}
  \frac{n-2}{n-1} x^3 + \left( \frac{1}{a+2} - a - \frac{n}{n-1} \right)x^2 - \frac{2(a+1)}{a+2} x + \frac{(a+1)^2}{a+2}=0.
\end{equation}
Observe that when $n = 2$, this reduces to the previous case.

When $n > 2$, in order to solve the cubic equation, we notice that $x = -1$ is a solution. We are not interested in this solution as this implies choosing $\gamma=-\mu$ which does not give ellipticity even for $p=2$. After long division and using quadratic formula, we obtain the other two solutions:

\begin{eqnarray}\nonumber
  x &=& \frac{-(n-1)(a+1)(a+3) \pm (a+1) \sqrt{ (n-1)^2 a^2 + 2(n+1)(n-1)a + (n+7)(n-1) }}{-2(n-2)(a+2)} \\
    &=& \frac{n-1}{2(n-2)}\frac{a+1}{a+2} \left( (a+3) \pm \sqrt{a^2 + \frac{2(n+1)}{n-1}a + \frac{n+7}{n-1} } \right) \\
    &=& \frac{n-1}{2(n-2)} \frac{\lambda + \mu }{\mu (\lambda + 2\mu)} \left( (\lambda + 3\mu) \pm \sqrt{\lambda^2 + \frac{2(n+1)}{n-1} \lambda \mu + \frac{n+7}{n-1} \mu^2} \right) \nonumber\\
    &=& \frac{n-1}{2(n-2)} \frac{\lambda + \mu }{\mu (\lambda + 2\mu)} \left( (\lambda + 3\mu) \pm \sqrt{(\lambda + \mu)^2 + \frac{4\mu(\lambda + 2\mu )}{n-1}} \right).\nonumber
\end{eqnarray}
Since we assume $\lambda > -2\mu$, both $\lambda + 3\mu$ and $\lambda + 2\mu$ are positive, therefore we consider the  positive solution. 

\begin{eqnarray}\label{EllipLp10b}
   \left( 1 - \frac{2}{p} \right)^2 
     &<
     1-\left(1+\frac1{n-2}\right)^2\left(\frac{\lambda+\mu}{\mu(\lambda+2\mu)}\right)^2\left[\frac{\lambda+3\mu}2-\sqrt{\left(\frac{\lambda+\mu}2\right)^2+\frac{\mu(\lambda+2\mu)}{n-1}}\right]^2,
\end{eqnarray}
then \eqref{EllipLp2} holds for some $r\in(0,2\mu)$.

\subsection{Necessary condition.}

If the system  \eqref{Lame}  satisfies the integral $p$-ellipticity condition then, for some $C>0$,
\begin{equation}\label{EllipLHp-LM}
\left\langle
(A^{hk}(x)q_h q_k)\xi,\xi\right\rangle-\left(1-\frac2p\right)^2\left\langle
(A^{hk}(x)q_h q_k)\omega\langle\omega,\xi\rangle,\omega\langle\omega,\xi\rangle\right\rangle \ge C |\xi|^{2}|q|^2,
\end{equation}
must hold for all $\xi,\omega \in{\mathbb R}^{m}$, $q\in{\mathbb R}^{n}$ with $|\omega|=1$, and a.e.  $x\in\Omega$ by (ii) of Theorem \ref{p-ellipticity AT}. By \eqref{eq-Lame-coeff-m} we have that 
\eqref{EllipLHp-LM} (dropping dependence on $x$) can be written as
\begin{equation}
    \mu |\xi|^2|q|^2 + (\lambda +\mu)\langle q,\xi\rangle^2 -\left( 1 - \frac{2}{p} \right)^2\langle\omega,\xi\rangle^2 \left[ \mu |q|^2 + (\lambda +\mu)\langle q,\omega\rangle^2\right]\ge C|\xi|^2|q|^2.
\end{equation}
Observe that this does not depend on the choice of $r$ in \eqref{eq-Lame-coeff-m}. Hence we must have
\begin{eqnarray}\label{EllipLp12}
   \left( 1 - \frac{2}{p} \right)^2 
     &<&\mbox{ess} \inf_{x\in\Omega} \inf_{|\omega|=|\xi|=|q|=1} \frac{ \mu(x)|\xi|^2|q|^2 + (\lambda(x) +  \mu(x))\langle q,\xi\rangle^2 }{\langle\omega,\xi\rangle^2 ( \mu(x)|q|^2 + (\lambda(x) +  \mu(x))\langle q,\omega\rangle^2)  }.
\end{eqnarray}

Here we have used the fact that in \eqref{EllipLp12} both the numerator and denominator scale identically with respect to $|\xi|$ and $|q|$ and so we may assume that $|\xi|=|q|=1$. Since the set $|\omega|=|\xi|=|q|=1$ is compact, the minimiser of \eqref{EllipLp12} exists for a fixed $x\in\Omega$. As in the previous section, if $(\omega,q,\xi)$ is one such minimiser then so is $(R\omega,Rq,R\xi)$ for any $R\in SO(n)$ and therefore we might as well assume that $q=e_1$. Hence the condition \eqref{EllipLp12} simplifies to
\begin{eqnarray}\label{EllipLp13}
   \left( 1 - \frac{2}{p} \right)^2 
     &<&\inf_{|\omega|=|\xi|=1} \frac{ \mu|\xi|^2 + (\lambda +  \mu)(\xi_1)^2 }{(\sum\omega_i\xi_i)^2 ( \mu + (\lambda +  \mu)(\omega_1)^2)  }.
\end{eqnarray}
Since $|\omega_1|\le 1$ we see that $\mu + (\lambda +  \mu)(\omega_1)^2>0$. Notice that the numerator does not depend on $\omega$ and $\omega_2,\omega_3,\dots$ only appear in the term 
$(\sum\omega_i\xi_i)^2$. We want to pick $\omega_2,\omega_3,\dots$ to maximise this term.

Let $A=\sum_{k>1}\omega_i\xi_i$. Then
$$(\sum \omega_i\xi_i)^2=(\omega_1\xi_1+A)^2 ,\qquad\mbox{and}\quad
|A|^2\le (1-(\omega_1)^2)(|\xi|^2-(\xi_1)^2)$$
by the Cauchy-Schwarz inequality. Hence if we set $(\tilde\omega_1,\tilde\omega_2)=(|\omega_1|,\sqrt{1-(\omega_1)^2})$ and $(\tilde\xi_1,\tilde\xi_2)=(|\xi_1|,\sqrt{|\xi|-(\xi_1)^2})$ we see that
\begin{eqnarray}\label{EllipLp14}
 && \frac{ \mu|\xi|^2 + (\lambda +  \mu)(\xi_1)^2 }{(\sum\omega_i\xi_i)^2 ( \mu + (\lambda +  \mu)(\omega_1)^2)  }\ge \frac{ \mu(\tilde\xi_2)^2 + (\lambda +  2\mu)(\tilde\xi_1)^2 }{(\tilde\omega_1\tilde\xi_1+\tilde\omega_2\tilde\xi_2)^2 ( \mu + (\lambda +  \mu)(\tilde\omega_1)^2)  },
\end{eqnarray}
so that \eqref{EllipLp13} holds if 
\begin{eqnarray}\label{EllipLp132d}
   \left( 1 - \frac{2}{p} \right)^2 
     &<&\inf_{|\tilde\omega|=|\tilde\xi|=1} \frac{ \mu(\tilde\xi_2)^2 + (\lambda +  2\mu)(\tilde\xi_1)^2 }{(\tilde\omega_1\tilde\xi_1+\tilde\omega_2\tilde\xi_2)^2 ( \mu + (\lambda +  \mu)(\tilde\omega_1)^2)  }.
\end{eqnarray}
But now $\tilde\xi,\,\tilde\omega\in\mathbb R^2$ and we have reduced the original problem to two dimensions. Let $C_p=(1-2/p)^2$. We rewrite the inequality above as

\begin{eqnarray}\nonumber
0&<& \left[\lambda +  2\mu-C_p\left[\mu+(\lambda +  \mu)(\tilde\omega_1)^2\right](\tilde\omega_1)^2
\right](\tilde\xi_1)^2+\left[\mu-C_p\left[\mu+(\lambda +  \mu)(\tilde\omega_1)^2\right](\tilde\omega_2)^2
\right](\tilde\xi_2)^2\\\label{EllipLp14}
&&\qquad-2C_p\left[\mu+(\lambda +  \mu)(\tilde\omega_1)^2\right]\tilde\omega_1\tilde\omega_2\tilde\xi_1\tilde\xi_2,
\end{eqnarray}
which is a quadratic form in $\tilde\xi$. In order for this form to be positive definite we have to satisfy
\begin{eqnarray}\label{EllipLp15}
0&<&\lambda +  2\mu-C_p\left[\mu+(\lambda +  \mu)(\tilde\omega_1)^2\right](\tilde\omega_1)^2,\\
\left[C_p\left[\mu+(\lambda +  \mu)(\tilde\omega_1)^2\right]\tilde\omega_1\tilde\omega_2\right]^2&<&
\nonumber \left[\lambda +  2\mu-C_p\left[\mu+(\lambda +  \mu)(\tilde\omega_1)^2\right](\tilde\omega_1)^2\right]\times\\\nonumber &&\qquad\qquad\qquad\left[\mu-C_p\left[\mu+(\lambda +  \mu)(\tilde\omega_1)^2\right](\tilde\omega_2)^2
\right],
\end{eqnarray}
for any $|\tilde\omega|=1$. To simplify the notation further let $\gamma_p=C_p\left[\mu+(\lambda +  \mu)(\tilde\omega_1)^2\right]$ and rewrite \eqref{EllipLp15} as
\begin{eqnarray}\label{EllipLp15a}
0&<&\lambda +  2\mu-\gamma_p(\tilde\omega_1)^2,\\
\gamma_p^2\left[\tilde\omega_1\tilde\omega_2\right]^2&<&
\nonumber \left[\lambda +  2\mu-\gamma_p(\tilde\omega_1)^2\right]
\left[\mu-\gamma_p(\tilde\omega_2)^2
\right].
\end{eqnarray}
The second condition can be rewritten as
$$(\lambda+2\mu)\mu-C_p\left[\mu+(\lambda +  \mu)(\tilde\omega_1)^2\right][\mu+(\lambda+\mu)(\tilde\omega_2)^2]>0,$$
for any $\tilde\omega\in\mathbb R^2$ with $|\tilde\omega|=1$.  The minimum of the lefthand side is attained when $|\tilde\omega_1|=|\tilde\omega_2|=1/\sqrt{2}$. It follows that the necessary condition is that
$$(\lambda+2\mu)\mu-C_p\left[\mu+(\lambda +  \mu)/2\right]^2>0.$$
Hence
\begin{equation}
\left(1-\frac2p\right)^2=C_p<\frac{(\lambda+2\mu)\mu}{(\mu+(\lambda +  \mu)/2)^2}=\frac{(2\lambda+4\mu)2\mu}{(\lambda+3\mu)^2}=\frac{(\lambda+3\mu)^2-(\lambda+\mu)^2}{(\lambda+3\mu)^2}.
\end{equation}
It follows that in all dimensions the necessary condition on $\lambda$ and $\mu$ such that the
integral $p$-ellipticity holds is that
\begin{eqnarray}\label{EllipLp11b}
   \left( 1 - \frac{2}{p} \right)^2 
     &<&1-\left(\frac{\lambda+\mu}{\lambda+3\mu}\right)^2,
\end{eqnarray}
Compare this to \eqref{EllipLp11}. Recall that we have dropped the dependance of $\mu$ and $\lambda$ on $x\in\Omega$ and hence \eqref{EllipLp11b} should interpreted as the condition that should hold for every $x\in\Omega$, i.e., the range of $p$ is determined as the infimum over the righthand side of \eqref{EllipLp11b} as $x$ varies over our domain.

It follows that when $n=2$ the necessary and sufficient conditions for the integral $p$-ellipticity condition to hold are the same. For $L^p$-dissipativity this is a known result (c.f Theorem 3.3 of Chapter 3 in \cite{CM17}). Our necessary condition we have proven here is not actually new, see for example again \cite{CM17}) where it is established by other methods in all dimensions.

We summarize the results of the last two subsections in the following theorem. 

\begin{theorem}\label{Prop-Lame} The Lam\'e system 
\begin{equation}\label{Lamedx}
\mathcal Lu=\nabla\cdot\left(\lambda(x)(\nabla\cdot u)I+\mu(x)(\nabla u+(\nabla u)^T) \right)=0
\end{equation}
for an unknown function $u:\Omega\to{\mathbb R}^n$ can be written in an equivalent form as
\begin{equation}\label{Lamexx}{\mathcal L}'u=0,\quad\mbox{where}\quad
\mathcal{L'}u=\left[ \partial_{i} \left({A}_{ij}^{\alpha \beta}(r)(x) \partial_{j} u_{\beta}\right)
+{B}_{i}^{\alpha \beta}(r)(x) \partial_{i}u_{\beta}\right]_{\alpha},
\end{equation}
where the Lipschitz function $r$ was introduced in \eqref{eq-Lame-coeff-m} and
gave rise to coefficients ${B}$ satisfying a simple estimate $|B|\lesssim |\nabla\lambda|+|\nabla\mu|$. Assume that 
$$\mbox{\rm ess }\inf_{x\in\Omega}\{\mu(x),\lambda(x)+2\mu(x)\}>0.$$

The coefficients $A$ of the operator $\mathcal L'$ satisfy the
integral $p$-ellipticity condition \eqref{eq:pellintp} if and only if
\begin{equation}\label{eq-LameIC}
\left(1-\frac2p\right)^2<C(n,\lambda,\mu),
\end{equation}
where
\begin{equation}\label{eq-LameICMa}
C(2,\lambda,\mu)=1-\mbox{\rm ess }\sup_{x\in\Omega}\left(\frac{\lambda+\mu}{\lambda+3\mu}\right)^2,\quad\mbox{for }n=2,
\end{equation}
and for $n\ge 3$ we have estimates for $C(n,\lambda,\mu)$ from above and below by
\begin{equation}\label{eq-LameICMb}
C(n,\lambda,\mu)\le 1-\mbox{\rm ess }\sup_{x\in\Omega}\left(\frac{\lambda+\mu}{\lambda+3\mu}\right)^2,
\end{equation}
\begin{eqnarray}\label{eq-LameICMc}
C(n,\lambda,\mu)&\ge& 1-\mbox{\rm ess }\sup_{x\in\Omega} \textstyle \left(1+\frac1{n-2}\right)^2\left(\frac{\lambda+\mu}{\mu(\lambda+2\mu)}\right)^2\left[\frac{\lambda+3\mu}2-\sqrt{\left(\frac{\lambda+\mu}2\right)^2+\frac{\mu(\lambda+2\mu)}{n-1}}\right]^2\\
&\ge&1-\mbox{\rm ess }\sup_{x\in\Omega}\left(\frac{\lambda+\mu}{\max\{\mu,\lambda+2\mu\} }\right)^2.\label{eq-LameICMd}
\end{eqnarray}
\end{theorem}


\subsection{Extrapolation of the Lam\'e system.}

We combine Theorem \ref{Prop-Lame} with our extrapolation result (Theorem \ref{Extrapolation}). 
Recall that the oscillation of a real function $f$ over a set $A$ (denoted by $\mbox{osc}_{A}\,f$) is defined by
$$\mbox{\rm osc}_{A}\, f=\sup_Af-\inf_Af.$$
We have the following theorem.

\begin{theorem} Let $\Omega$ be a bounded or unbounded Lipschitz domain. Consider the 
$L^p$ Dirichlet problem for the Lam\'e system
  \begin{equation}\label{LMDP}
    \begin{cases}
      \mathcal Lu=\nabla\cdot\left(\lambda(x)(\nabla\cdot u)I+\mu(x)(\nabla u+(\nabla u)^T) \right)=0. \text{ on } \Omega, \\
      u(x) = f(x) \in L^p (\partial \Omega)\qquad \text{ for } \sigma-a.e. \ x \in \partial\Omega, \\
      \tilde{N}_{2,a}(u) \in L^p (\partial \Omega).
          \end{cases} 
  \end{equation}   
Assume that 
$$\mbox{\rm ess }\inf_{x\in\Omega}\{\mu(x),\lambda(x)+2\mu(x)\}>0,$$
and with $C(n,\lambda,\mu)$ as in Theorem \ref{Prop-Lame} set $p_0=\frac{2}{(1-C(n,\lambda,\mu))^{1/2}}$. Assume that for some
$1<q<\frac{p_0(n-1)}{(n-2)}$ the $L^q$ Dirichlet problem \eqref{LMDP} is solvable. 

Then for every $p\in \left(q,\frac{p_0(n-1)}{(n-2)}\right)$ there exists $K(p)>0$ of the following significance. If
\begin{equation}\label{eq-osc}
\mbox{\rm osc}_{B(x,\delta(x)/2)}\,\lambda+\mbox{\rm osc}_{B(x,\delta(x)/2)}\,\mu\le K(p)\qquad\forall x\in\Omega,
\end{equation}
then the $L^p$ Dirichlet problem for the Lam\'e system
\eqref{LMDP} is solvable. If $\Omega$ is a bounded domain, then the condition \eqref{eq-osc} only has to hold for points with $\delta(x)\le \delta_0$ for some $\delta_0>0$.
\end{theorem}

\begin{proof} Observe that  $p_0$ as defined here corresponds to $p_0$ as defined in Theorem \ref{Extrapolation} due to Theorem \ref{Prop-Lame}.
Suppose that $\lambda,\mu$ are as above. Consider a pair of mollified Lam\'e coefficients
\begin{equation}\label{eq-lmu}
\tilde{\lambda}(x)= \int_{{\mathbb R}^n} \lambda(y)\varphi_{\rho(x)}(x-y)dy,\quad \tilde{\mu}(x)= \int_{{\mathbb R}^n} \mu(y)\varphi_{\rho(x)}(x-y)dy,
\end{equation}
for $x\in\Omega$.
Here $\varphi$ is a smooth real, nonnegative bump function on ${\mathbb R}^n$ supported in the ball  $B_{1/2}(0)$ such that $\int \varphi=1$ and
$\varphi_t(y) = t^{-n}\varphi(y/t)$. By $\rho(x)$ we denote a mollified distance function (i.e. $\rho(x)\approx\delta(x)$ but $\rho$ is smooth in the interior of $\Omega$). 

If follows that $\tilde\lambda,\,\tilde\mu$ are differentiable in $\Omega$ with
\begin{equation}\label{eq-88}
|\nabla\tilde\lambda(x)|+|\nabla\tilde\mu(x)|\lesssim \frac{K(p)}{\delta(x)},
\end{equation}
where $K(p)$ is as in \eqref{eq-osc}. Additionally, we also have
\begin{equation}\label{eq-221}
\sup_{x\in\Omega}|\lambda-\tilde\lambda|+\sup_{x\in\Omega}|\mu-\tilde\mu|\lesssim K(p).
\end{equation}

Fix some $s$ such that $(1-2/s)^2<C(n,\lambda,\mu)$. As follows from Proposition \ref{Prop-Lame}  there exists some bounded function $r(x)$ such that with coefficients ${A}$ as in \eqref{eq-Lame-coeff-m}
we have for some $C>0$
\begin{equation}\label{eq:pellintp2}
 \int_\Omega \left\langle A(r) \left(\nabla v-\left(1-\frac2s\right)\frac{v}{|v|}\nabla|v|\right) , \nabla v +\left(1-\frac2s\right)\frac{v}{|v|}\nabla|v|\right\rangle dx \geq C \int_\Omega |\nabla v|^2\, dx,
  \end{equation}
for all $v\in W^{1,2}_0(\Omega,\mathbb R^n)$.
Observe that we have an issue with the coefficients $B$ defined as in \eqref{eq-Lame-coeff-m} as our $r$ might not be differentiable. This is where the mollified $\tilde\lambda,\,\tilde\mu$ come into the play.
Let 
\begin{eqnarray}\label{eq-Lame-coeff-2}
  &&\widetilde{A^{hk}_{\alpha \beta}}(r) (x) = \tilde\mu(x) \delta^{hk} \delta_{\alpha \beta} + ( \tilde\lambda(x) + r(x) )\delta^h_\alpha \delta^k_\beta + ( \tilde\mu(x) - r(x) ) \delta^h_\beta \delta^k_\alpha.
\end{eqnarray}  
Observe that by  \eqref{eq-221} we see that $|A(r)-\widetilde{A}(r)|\lesssim K(p)$ and hence by \eqref{eq:pellintp2}
\begin{equation}\nonumber
\int_\Omega \left\langle \widetilde{A}(r) \left(\nabla v-\left(1-\frac2s\right)\frac{v}{|v|}\nabla|v|\right) , \nabla v +\left(1-\frac2s\right)\frac{v}{|v|}\nabla|v|\right\rangle dx \geq C(1-n^4\|A-\widetilde{A}\|_{L^\infty}) \int_\Omega |\nabla v|^2\, dx,
  \end{equation}
which implies that for sufficiently small $K(p)>0$ we get that the operator with coefficients $\widetilde{A}(r)$ satisfies the integral condition  \eqref{eq:pellintp} and hence 
$(1-2/s)^2<C(n,\tilde\lambda,\tilde\mu)$. Therefore Proposition \ref{Prop-Lame} can be applied to the pair $\tilde\lambda,\,\tilde\mu$ and there exists a new function $\tilde{r}$ such that
$|\nabla\tilde r|\lesssim |\nabla\tilde\lambda|+|\nabla\tilde\mu|$ and
\begin{equation}\nonumber
\int_\Omega \left\langle \widetilde{A}(\tilde{r}) \left(\nabla v-\left(1-\frac2s\right)\frac{v}{|v|}\nabla|v|\right) , \nabla v +\left(1-\frac2s\right)\frac{v}{|v|}\nabla|v|\right\rangle dx \geq C_2 \int_\Omega |\nabla v|^2\, dx,
  \end{equation}
holds for some $C_2>0$. Thus, by the same argument as above if $K(p)$ is small enough we might achieve that
\begin{equation}\label{eq:pellintp2b}
 \int_\Omega \left\langle A(\tilde{r}) \left(\nabla v-\left(1-\frac2s\right)\frac{v}{|v|}\nabla|v|\right) , \nabla v +\left(1-\frac2s\right)\frac{v}{|v|}\nabla|v|\right\rangle dx \geq \frac{C_2}2 \int_\Omega |\nabla v|^2\, dx.
  \end{equation}
However, now $\tilde{r}$ is differentiable. Set 
\begin{equation}
{B}^h_{\alpha \beta}(\tilde{r}) (x) = \partial_k \tilde{r}(x) (\delta^h_\alpha \delta^k_\beta - \delta^h_\beta \delta^k_\alpha  ).
 \end{equation}
 Clearly, by \eqref{eq-88} we have that $| B(\tilde r)|\lesssim \frac{K(p)}{\delta(x)}$ and
 for 
$$ \mathcal{L'}u=\left[ \partial_{i} \left({A}_{ij}^{\alpha \beta}(\tilde{r}(x))(x) \partial_{j} u_{\beta}\right)
+{B}_{i}^{\alpha \beta}(\tilde{r}(x))(x) \partial_{i}u_{\beta}\right]_{\alpha},$$
we have $\mathcal L'u=0$ iff $\mathcal Lu=0$. Hence, by making $K(p)$ smaller if necessary we can ensure that Theorem \ref{Extrapolation} applies implying solvability of the $L^p$ Dirichlet problem \eqref{LMDP} for all $q<p<s(n-1)/(n-2)$. From this the claim follows.
\end{proof}

It was shown in \cite[Corollary 1.5]{D20} that the $L^2$ Dirichlet and regularity problems are solvable for the Lam\'e system under the assumption that $\lambda,\,\mu$ satisfy a certain Carleson measure \eqref{Car_lame2} with small constants, under the ellipticity assumption \eqref{Cond-lame}. We can therefore draw the following corollary of this last result.

\begin{corollary} \label{C:lame}
Let $\Omega$ be the Lipschitz domain $\{(x_0,x')\in{\mathbb{R}}\times{\mathbb{R}}^{n-1}:\,x_0>\phi(x')\}$ 
with Lipschitz constant $L=\|\nabla\phi\|_{L^\infty}$. Assume that the Lame coefficients $\lambda,\mu\in L^\infty(\Omega)$ satisfy the following:
\begin{itemize} 
\item[$(i)$] There exists $\mu_0>0$ such that
\begin{equation}\label{Cond-lame}
\mbox{\rm ess }\inf_{x\in\Omega}\{(\sqrt{8}-1)\mu(x)+\lambda(x),(\sqrt{8}+1)\mu(x)-\lambda(x)\}\ge \mu_0.
\end{equation}\item[$(ii)$]
\begin{equation}\label{Car_lame2}
d{\nu}(x)=\left[\left(\osc_{B_{\delta(x)/2}(x)}{\lambda}\right)^{2}+\left(\osc_{B_{\delta(x)/2}(x)}{\mu}\right)^{2}\right]\delta^{-1}(x)
\end{equation}
is a Carleson measure in $\Omega$. 
\end{itemize}
With $C(n,\lambda,\mu)$ as in Proposition \ref{Prop-Lame} consider any 
$2-\varepsilon<p<\frac{2(n-1)}{(n-2)(1-C(n,\lambda,\mu))^{1/2}}$.

Then there exist $\varepsilon=\varepsilon(\mu_0,\|\lambda\|_{L^\infty},\|\mu\|_{L^\infty},n)>0$ and  $K=K(\mu_0,\|\lambda\|_{L^\infty},\|\mu\|_{L^\infty},n,p)>0$ such that if
\begin{equation}\label{Small-Cond2}
\max\big\{L\,,\,\|\nu\|_{\mathcal C}\big\}\leq K
\end{equation}
then $L^p$-Dirichlet problem \eqref{LMDP} for the Lam\'e system is solvable and
the estimate
\begin{equation}\label{Main-Est-LM}
\|\tilde{N}_{p,a} u\|_{L^{p}(\partial \Omega)}\leq C\|f\|_{L^{p}(\partial \Omega;{\mathbb R}^n)}
\end{equation}
holds for all energy solutions $u:\Omega\to {\mathbb R}^n$ with datum $f$. Here $C=C(\mu_0,\|\lambda\|_{L^\infty},\|\mu\|_{L^\infty},n,p)>0$. 
\end{corollary}

\noindent{\it Remark.}  Observe that the condition \eqref{Cond-lame} implies that
$$1-\sqrt{8}<\lambda/\mu<1+\sqrt{8},$$
which by \eqref{eq-LameICMc} implies a particular lower bound on the value of $C(n,\lambda,\mu)$. In particular, it implies that the solvability range $p\in (2-\varepsilon,p(n))$ in the above Corollary for the $L^p$ Dirichlet problem is at least:
$$p(2)=\infty,\qquad p(3)>11.50,\quad p(4)>8.055,\qquad p(n)>\frac{2(n-1)}{(n-2)(1-\sqrt{8\sqrt{2}-11})}\approx \frac{4.546(n-1)}{n-2}.$$\vglue2mm

\noindent{\it Remark 2.} We first recall that, under the assumptions of Corollary \ref{C:lame}, the Regularity boundary value problem is solvable \cite{D20}. The extrapolation ideas in \cite{S3}, explicitly stated in \cite{D20},  show that the solvability of the Regularity problem yields a further improvement on the range in $L^p$ of solvability of the Dirichlet problem, beyond what one can achieve without Regularity. However, our extrapolation result assuming $p$-ellipticity goes even further 
as soon as the dimension $n=4$. Our $p(4) > 8$, while using the assumption of solvability of Regularity one would only find $p(4)=6$.
  
\smallskip

\noindent{\it Remark 3.}  Regarding the condition \eqref{Cond-lame}, it was observed in \cite{DHM} that the physical constraints
for certain typical materials imply that $\mu>-\frac{2}{n}\lambda$. The constant $K=\mu+\frac2n\lambda$ is the {\it bulk modulus}
and is positive; it is defined as the ratio of the infinitesimal pressure increase to the resulting relative decrease of the volume.
Hence our condition \eqref{Cond-lame} only imposes one additional assumption, namely that 
$$\lambda<(\sqrt{8}+1)\mu\approx 3.828\mu,$$
or alternatively the Poisson ratio $\nu:=\frac{\lambda}{2(\lambda+\mu)}<0.396$. There are many materials where this holds (for example  aluminium, bronze, steel and many other metals, carbon,  polystyrene, PVC, silicate glasses, concrete, etc) \cite{MR}. Some materials where this assumption fails include gold, lead or rubber. For these three materials $\nu$ is near the incompressibility limit ($\nu=\frac12-$) at which  \eqref{LMDP} gives {\rm div}$\,u=0$, i.e., the material is incompressible. Intuitively, since both gold and lead are very soft metals, they behave as liquids under pressure; that is, a pressure in one direction will cause them to change shape and stretch in remaining directions in order to preserve volume. Rubber is nearly incompressible with $\nu\approx 0.49$.

\section{Application to periodic homogenization of elliptic systems}\label{six}

In this section we will consider a family of second order real elliptic systems with rapidly oscillating periodic coefficients, i.e.

\begin{equation}
  (\mathcal{L}_\epsilon u)_\alpha = \partial_h (A^{hk}_{\alpha \beta}(x / \epsilon) \partial_k u^\beta), \ \epsilon > 0,
\end{equation}
on $\mathbb{R}^n$. We further assume the coefficient matrix

\begin{equation}
  A(x) = (A_{\alpha \beta}^{hk} (x)), \ 1 \leq h,k \leq n \text{ and } 1 \leq \alpha,\beta \leq m
\end{equation}
is 1-periodic, i.e. $A(x + y) = A(x)$ for a.e. $x \in \mathbb{R}^n$ and $y \in \mathbb{Z}^n$. We call $A \in \Lambda(C_1, C_2, \tau )$ if $A$ is 1-periodic, $A = A^*$ and satisfies the strong ellipticity condition (Legendre condition):

\begin{equation}
  C_1 |\xi|^2 \leq \langle A(x) \xi , \xi \rangle \leq \frac{1}{C_1} |\xi|^2 \text{ for any } \xi= (\xi_h^\alpha) \in \mathbb{R}^{n \times m} 
\end{equation}
and also satisfies the H\"older continuity condition:

\begin{equation}
  |A(x) - A(y)| \leq C_2 |x - y|^\tau, \ \tau \in (0,1]. 
\end{equation}
In this theory, it is of interest to consider the existence of a uniform estimate for the $L^p$-Dirichlet problem on bounded Lipschitz domains:

\begin{equation}
  \begin{cases}
    \mathcal{L}_\epsilon u_\epsilon = 0 &\text{ on } \Omega, \\
    u_\epsilon = f \in L^p(\partial \Omega; \mathbb{R}^m) &\text{ on } \partial \Omega, \\
    N_a(u_\epsilon) \in L^p (\partial \Omega).
  \end{cases}
\end{equation}
For the case $p = 2$, it was proved in \cite{KS} that there exists a unique solution $u_\epsilon \in C^1 (\Omega; \mathbb{R}^m)$ satisfying 

\begin{equation}
  ||N_a(u_\epsilon)||_{L^2(\partial \Omega)} \leq C ||f||_{L^2(\partial \Omega)},
\end{equation}
where $C$ depends only on $C_1,C_2,\tau$ and the Lipschitz constant of $\Omega$. We will prove the following:

\begin{theorem}\label{Homogen}
  Let $\Omega$ be a bounded Lipschitz domain, suppose $A \in \Lambda(C_1,C_2,\tau)$ and denote $q = \sup \{ A(x) \text{ satisfies condition } \eqref{EllipLp} \}$ . Then the $L^p$-Dirichlet problem 
  
  \begin{equation}
  \begin{cases}
    \mathcal{L}_\epsilon u_\epsilon = 0 &\text{ on } \Omega, \\
    u_\epsilon = f \in L^p(\partial \Omega; \mathbb{R}^m) &\text{ on } \partial \Omega, \\
    N_a(u_\epsilon) \in L^p (\partial \Omega).
  \end{cases}
  \end{equation}
is solvable for $2 < p < \frac{q(n-1)}{n-2}$ . Moreover, there exists $C = C(C_1,C_2,\tau,m,n,p,||A||_{L^\infty})$ such that

\begin{equation}
  ||N_a(u_\epsilon)||_{L^p(\partial \Omega)} \leq C ||f||_{L^p(\partial \Omega)}.
\end{equation}
   
\end{theorem} 

\begin{proof}
To obtain the uniform estimate, we first notice that since $A$ is strongly elliptic, from Theorem \ref{p-ellipticity AT}, there exists a small neighborhood of $2$ in which $A$ is strongly $p$-elliptic. For simplicity, we denote this optimal range as $(q',q)$. Moreover, from part (v) of Theorem \ref{p-ellipticity AT}, the value of $q$ depends on $||A||_{L^\infty}$ and this implies $A(x/\epsilon)$ is strongly $p$-elliptic in the same range. Since there are no lower order terms in this setting, Theorem \ref{Extrapolation} can be applied directly and it suffices to show that the constant in the $L^p$ estimate upon extrapolation is independent of $\epsilon$. To see this, from the proof of Theorem \ref{Extrapolation}, we need to show that the constant in \eqref{reverse holder of N} is independent of $\epsilon$. From the interior and boundary estimates in Section 2, this constant depends only on $m,n,C_1,C_2,p$, completing the proof.
\end{proof}

\begin{remark}
  When $m = 1$ or $n = 2,3$, the $L^p$-Dirichlet problem is solvable for $2-\delta < p < \infty$ and for $m \geq 2$ and $n \geq 4$, the extrapolation can be established for $2-\delta < p < \frac{2(n-1)}{n-3}+ \delta$ from \cite{S3}.
\end{remark}

\end{document}